\documentclass{amsart}
\usepackage{amscd,amsmath,amssymb,amsfonts}
\usepackage[cmtip, all]{xy}
\usepackage[OT2,T1]{fontenc}
\newcommand\textcyr[1]{{\fontencoding{OT2}\fontfamily{wncyr}\selectfont #1}}

\newtheorem{thm}{Theorem}[section]
\newtheorem{prop}[thm]{Proposition}
\newtheorem{lem}[thm]{Lemma}
\newtheorem{cor}[thm]{Corollary}

\renewcommand{\theclaim}{\kern-3pt}

\newtheorem{IntroThm}{Theorem}
\newtheorem{IntroProp}[IntroThm]{Proposition}

\newtheorem{IntroConj}[IntroThm]{Conjecture}

\theoremstyle{definition}
\newtheorem{Def}[thm]{Definition}

\theoremstyle{remark}
\newtheorem{rem}[thm]{Remark}
\newtheorem{rems}[thm]{Remarks}
\newtheorem{IntroRem}{Remark}
\newtheorem{IntroRems}[IntroRem]{Remarks}
\newtheorem{ex}[thm]{Example}

\numberwithin{equation}{section}

\newcommand{\sA}{{\mathcal A}}
\newcommand{\sB}{{\mathcal B}}
\newcommand{\sC}{{\mathcal C}}

\newcommand{\sE}{{\mathcal E}}
\newcommand{\sF}{{\mathcal F}}
\newcommand{\sG}{{\mathcal G}}
\newcommand{\sH}{{\mathcal H}}

\newcommand{\sK}{{\mathcal K}}

\newcommand{\sO}{{\mathcal O}}

\newcommand{\sR}{{\mathcal R}}
\newcommand{\sS}{{\mathcal S}}
\newcommand{\sT}{{\mathcal T}}

\newcommand{\sX}{{\mathcal X}}

\newcommand{\A}{{\mathbb A}}

\newcommand{\F}{{\mathbb F}}
\newcommand{\G}{{\mathbb G}}

\newcommand{\N}{{\mathbb N}}
\renewcommand{\P}{{\mathbb P}}
\newcommand{\Q}{{\mathbb Q}}
\newcommand{\R}{{\mathbb R}}
\newcommand{\mS}{{\mathbb S}}

\newcommand{\Z}{{\mathbb Z}}

\renewcommand{\phi}{\varphi}

\newcommand{\codim}{{\rm codim}}

\newcommand{\Hom}{{\rm Hom}}
\newcommand{\End}{{\rm End}}

\newcommand{\im}{{\rm im}}
\newcommand{\Spec}{\operatorname{Spec}}

\newcommand{\Char}{\operatorname{char}}
\newcommand{\Tr}{{\rm Tr}}
\newcommand{\Gal}{{\rm Gal}}

\newcommand{\0}{\emptyset}
\newcommand{\sHom}{{\mathcal{H}{om}}}

\newcommand{\id}{{\operatorname{id}}}

\newcommand{\op}{{\text{\rm op}}}
\newcommand{\<}{\mathopen<}
\renewcommand{\>}{\mathclose>}

\newcommand{\del}{\partial}

\renewcommand{\max}{{\operatorname{\rm max}}}

\newcommand{\Spt}{{\mathbf{Spt}}}
\newcommand{\Spc}{{\mathbf{Spc}}}
\newcommand{\Sm}{{\mathbf{Sm}}}

\newcommand{\hocolim}{\mathop{{\rm hocolim}}}
\renewcommand{\lim}{\operatornamewithlimits{\varprojlim}}
\newcommand{\colim}{\operatornamewithlimits{\varinjlim}}

\newcommand{\Ho}{{\mathbf{Ho}}}

\newcommand{\fin}{{\operatorname{\rm fin}}}
\newcommand{\SH}{{\operatorname{\sS\sH}}}

\newcommand{\eff}{{\mathop{eff}}}

\newcommand{\DM}{{DM}}
\newcommand{\GW}{\operatorname{GW}}

\newcommand{\Nis}{{\operatorname{Nis}}}

\newcommand{\ds}{{/\kern-3pt/}}

\newcommand{\hofib}{{\mathop{\rm{hofib}}}}

\newcommand{\et}{{\operatorname{\acute{e}t}}}

\newcommand{\Mod}{{\operatorname{Mod}}}

\newcommand{\s}{\tilde{s}}

\newcommand{\Tate}{\text{Tate}}
\renewcommand{\:}{\kern-1.5pt:\kern-1.5pt}

\newcommand{\Fil}{\operatorname{Fil}}
\newcommand{\gr}{\text{gr}}
\newcommand{\cd}{\operatorname{cd}}
\newcommand{\cfin}{{\operatorname{coh. fin}}}
\newcommand{\trdim}{\text{tr.\,dim}}

\begin{document}

\title{Convergence of Voevodsky's slice tower}
\author{Marc Levine}
\address{Universit\"at Duisburg-Essen\\
Fakult\"at Mathematik, Campus Essen\\
45117 Essen\\
Germany}
\email{marc.levine@uni-due.de}
\thanks{Research supported by the Alexander von Humboldt Foundation}


 
\renewcommand{\abstractname}{Abstract}
\begin{abstract}  
We consider Voevodsky's slice tower for a finite spectrum $\sE$ in the motivic stable homotopy category over a perfect field $k$. In case $k$ has finite cohomological dimension (in characteristic two, we also require that $k$ is infinite), we show that the slice tower converges, in that the induced filtration on the bi-graded homotopy sheaves  $\Pi_{a,b}f_n\sE$ is finite, exhaustive and separated at each stalk. This partially verifies a conjecture of Voevodsky. 
 \end{abstract}
\date{\today}
\maketitle
\tableofcontents

\section*{Introduction}  We continue our investigation, begun in \cite{LevineGW}, of the slice filtration on the bi-graded homotopy sheaves $\Pi_{*,*}(\sE)$ for objects $\sE$ in the motivic stable homotopy category $\SH(k)$. We refer the reader to \S\ref{sec:Background} for the notation to be used in this introduction.

Let $k$ be a perfect field, let $\SH(k)$ denote Voevodsky's motivic stable homotopy category of $T$-spectra over $k$, $\SH$ the classical stable homotopy category of spectra. For a spectrum $E\in\SH$, the  Postnikov (pre)tower of $E$, 
\[
\ldots\to E^{(n+1)}\to E^{(n)}\to \ldots\to E
\]
consists of the $n-1$-connected covers $E^{(n)}\to E$ of $E$, that is, $\pi_mE^{(n)}\to \pi_mE$ is an isomorphism for $m\ge n$ and  $\pi_mE^{(n)}=0$ for $m<n$. Sending $E$ to $E^{(n)}$ defines a functor from $\SH$ to the full subcategory $\Sigma^n\SH^\eff$ of $n-1$-connected spectra that is right adjoint to the inclusion  $\Sigma^n\SH^\eff\to \SH$.

Replacing  $\Sigma^n\SH^\eff$ with a certain triangulated subcategory $\Sigma^n\SH_T^\eff(k)$ of $\SH(k)$ that measures a kind of ``$\P^1$-connectedness'' (in a suitable sense, see \cite{VoevOpen, VoevSlice}, \cite{Pelaez, PelaezCR}) or \S\ref{sec:SliceTower} of this paper), Voevodsky has defined a motivic analog of the Postnikov tower; for  an object $\sE$ of $\SH(k)$ this yields the {\em Tate-Postnikov tower} (or slice tower)
\[
\ldots\to f_{n+1}\sE\to f_n\sE\to\ldots\to \sE
\]
for $\sE$. For integers $a,b$, we have the stable  homotopy sheaf $\Pi_{a,b}(\sE)$, defined as the Nisnevich sheaf associated to the presheaf
\[
U\in\Sm/k\mapsto [\Sigma^a_{S^1}\Sigma^b_{\G_m}\Sigma^\infty_TU_+,\sE]_{\SH(k)}
\]
(note that the indexing is not the standard one). The Tate-Postnikov tower for $\sE$ gives rise to the filtration
\[
\Fil^n_\Tate\Pi_{a, b}(\sE):=\im(\Pi_{a,b}f_n\sE\to \Pi_{a, b}\sE).
\]

Let $\SH_\fin(k)\subset \SH(k)$ be the thick subcategory of $\SH(k)$ generated by the objects $\Sigma_T^n\Sigma^\infty_TX_+$, with $X$ smooth and projective over $k$, $n\in\Z$. For example, the motivic sphere spectrum $\mS_k:=\Sigma^\infty_T\Spec k_+$ is in $\SH_\fin(k)$.
 
Voevodsky has stated the following conjecture:
\begin{IntroConj}[\hbox{\cite[conjecture 13]{VoevOpen}}]\label{Conj:Main} Let $k$ be a perfect field. Then for  $\sE\in \SH_\fin(k)$, the Tate-Postnikov tower of $\sE$ is convergent in the following sense: for all $a,b,n\in\Z$, one has
\[
\cap_mF^m_\Tate\Pi_{a,b}f_n\sE=0.
\]
\end{IntroConj}
The cases $\sE=\Sigma^q_{\G_m}\mS_k$, $a=n=0$ gives some evidence for this conjecture, as we shall now explain. 

For $k$ a perfect field, Morel has given a natural isomorphism of 
$\Pi_{0,-p}(\mS_k)$ with the {\em Milnor-Witt sheaf} $\sK^{MW}_{p}$; this is a certain sheaf on $\Sm/k$ with value on each field $F$ over $k$ given by the Milnor-Witt group $K^{MW}_{p}(F)$.\footnote{A presentation of the graded ring $K^{MW}_*(F)$ may be found in \cite[definition 3.1]{MorelA1}.}  For   $F$ a field,  $K^{MW}_{0}(F)$ is canonically isomorphic to the Grothendieck-Witt group $\GW(F)$ of non-degenerate symmetric bilinear forms over $F$ \cite[lemma 3.10]{MorelA1}.  More generally, Morel has constructed a natural isomorphism\footnote{This follows from  \cite[theorem 6.13, theorem 6.40]{MorelA1}, using the argument of \cite[theorem 6.43]{MorelA1}.} for $p, q\in\Z$
\[
\Pi_{0,p}(\Sigma^q_{\G_m}\mS_k)\cong \sK^{MW}_{q-p}.   
\]

The isomorphism $K^{MW}_{0}(F)\cong \GW(F)$ makes $K^{MW}_*(F)$ a $\GW(F)$-module; let $I(F)\subset \GW(F)$ denote the augmentation ideal. Our main result of {\it loc. cit.} is 
\begin{IntroThm}[\hbox{\cite[theorem 1]{LevineGW}}]\label{IntroThm:LevineGW} Let $F$ be a perfect field extension of $k$ of characteristic $\neq2$. Then 
\[
\Fil^n_\Tate\Pi_{0,p}(\Sigma^q_{\G_m}\mS_k)(F)= I(F)^MK^{MW}_{q-p}(F)\subset K^{MW}_{q-p}(F)=\Pi_{0,p}(\Sigma^q_{\G_m}\mS_k)(F)
\]
where $M=0$ if $n\le p$ or $n\le q$, and $M=\min(n-p, n-q)$ if $n\ge p$ and $n\ge q$.
\end{IntroThm}

The following consequence of theorem~\ref{IntroThm:LevineGW} gives some evidence for Voevodsky's convergence conjecture:
\begin{IntroProp}\label{IntroProp:converge}  Let $k$ be a perfect field with $\Char k\neq2$.  For all $p,q\ge0$, and all perfect field extensions $F$ of $k$, we have
\[
\cap_nF^n_\Tate\Pi_{0,p}\Sigma_{\G_m}^q\mS_k(F)=0.
\]
\end{IntroProp}

\begin{proof} In light of theorem~\ref{IntroThm:LevineGW},  the assertion is that the $I(F)$-adic filtration on $K^{MW}_{q-p}(F)$ is separated. By \cite[th\'eor\`eme 5.3]{MorelWitt}, for $m\ge0$, $K^{MW}_m(F)$ fits into a cartesian square of $\GW(F)$-modules
\[
\xymatrix{
K^{MW}_m(F)\ar[r]\ar[d]&K^M_m(F)\ar[d]^{Pf}\\
I(F)^m\ar[r]_-q&I(F)^m/I(F)^{m+1},
}
\]
where $K^M_m(F)$ is the Milnor $K$-group, $q$ is the quotient map and $Pf$ is the map sending a symbol $\{u_1,\ldots,u_m\}$ to the class of the Pfister form $\<\<u_1,\ldots,u_m\>\>$ mod $I(F)^{m+1}$.  For $m<0$, $K^{MW}_m(F)$ is isomorphic to the Witt group $W(F)$  of $F$,  that is, the quotient of $\GW(F)$ by the ideal generated by the hyperbolic form $x^2-y^2$. Also, the map $\GW(F)\to W(F)$ gives an isomorphism of $I(F)^r$ with its image in $W(F)$ for all $r\ge1$. Thus, for $n\ge1$, 
\[
I(F)^nK^{MW}_m(F)=\begin{cases}I(F)^n\subset W(F)&\text{ for }m<0\\
I(F)^{n+m}\subset  W(F)&\text{ for }m\ge0.
\end{cases}
\]
The fact that $\cap_nI(F)^n=0$ in $W(F)$ is a theorem of Arason and Pfister \cite[Korollar 1]{ArasonPfister}.
\end{proof}

\begin{IntroRems} 1. The proof in  \cite{MorelWitt} that $K^{MW}_m(F)$ fits into a cartesian square as above relies on the Milnor conjecture.\\
2. As pointed out to me by Igor Kriz, Voevodsky's convergence conjecture in the generality as stated above is false. In fact, take $\sE$ to be the Moore spectrum $\mS_k/\ell$ for some prime $\ell\neq 2$. Since $\Pi_{a,q}\mS_k=0$ for $a<0$,   proposition~\ref{prop:Vanishing1} below shows that $\Pi_{a,q}f_n\mS_k=0$ for $a<0$, and  thus we have the right exact sequence for all $n\ge0$
\[
\Pi_{0,0}f_n\mS_k\xrightarrow{\times \ell} \Pi_{0,0}f_n\mS_k\to  \Pi_{0,0}f_n\sE\to0.
\]
In particular, we have
\[
F^n_\Tate\Pi_{0,0}\sE(k)=im\left(F^n_\Tate\Pi_{0,0}\mS_k(k)\to \Pi_{0,0}\mS_k(k)/\ell\right)=im\left(I(k)^n\to \GW(k)/\ell\right).
\]
Take $k=\R$. Then $\GW(\R)=\Z\oplus\Z$, with virtual rank and virtual index giving the two factors. The augmentation ideal $I(\R)$ is thus isomorphic to $\Z$ via the index and it is not hard to see that $I(\R)^{n}=(2^{n-1})\subset\Z=I(\R)$. Thus $\Pi_{0,0}\sE=\Z/\ell\oplus \Z/\ell$ and the filtration $F^n_\Tate\Pi_{0,0}\sE$  is constant, equal to $\Z/\ell=I(\R)/\ell$, and is therefore not separated.

The convergence property is thus not a ``triangulated" one in general, and therefore seems to be a subtle one. However, if the $I$-adic filtration on $\GW(F)$ is finite for all finitely generated $F$ over $k$ (possibly of varying length depending on $F$), then the augmentation ideal in $\GW(F)$ is  two-primary torsion. Our computations (at least in characteristic zero) show that the filtration $F^*_\Tate\Pi_{0,p}\Sigma^\infty_T\G_m^{\wedge q}$ is in this case at least locally finite, and thus has better triangulated properties. In particular, for $\ell\neq2$,
\[
\Pi_{0,0}(\mS_k/\ell)=\Z/\ell,\ F^n_\Tate\Pi_{0,0}(\mS_k/\ell)=0\text{ for }n>0.
\]
One can therefore ask if Voevodsky's convergence conjecture is true if one assumes the finiteness of the $I(F)$-adic filtration on $\GW(F)$ for all finitely generated fields $F$ over $k$. The main theorem of this paper is a partial answer to the convergence question along these lines. 
\end{IntroRems}

 \begin{IntroThm}\label{IntroThm:Main} Let $k$ be a perfect field  of finite cohomological dimension and let $p$ denote the exponential characteristic.\footnote{That is,  $p=\Char k$ if $\Char k>0$, $p=1$ if $\Char k=0$.}   Take $\sE$ in  $\SH_\fin(k)$ and take $x\in X\in\Sm/k$ with $X$ irreducible. Let $d=\dim_kX$.  Then for every $r,q\in\Z$, there is an integer $N=N(\sE, r,d,q)$   such that 
\[
(\Fil_\Tate^n\Pi_{r,q}\sE)_x[1/p]=0
\]
for all $n\ge N$.  In particular, if $F$ is a field extension of $k$ of finite transcendence dimension $d$ over $k$, then $\Fil_\Tate^n\Pi_{r,q}\sE(F)[1/p]=0$ for all $n\ge N$. 
 \end{IntroThm} 
For a more detailed and perhaps more general statement, we refer the reader to theorem~\ref{thm:Main}.

\begin{IntroRems}1. The proof of theorem~\ref{IntroThm:Main} relies on the Bloch-Kato conjecture. \\
2.  As we have seen, Voevodsky's convergence conjecture is not true for all base fields $k$. An interesting class of fields strictly larger than the class of fields of finite cohomological dimension is those of finite virtual cohomological dimension (e.g., $\R$). We suggest the following formulation:
\begin{IntroConj} Let $k$ be a field of finite virtual 2-cohomological dimension and finite $p$-cohomological dimension for primes $p>2$. Then the $I(k)$-completed slice tower is convergent.
\end{IntroConj}\noindent
3.  It would be interesting to be able to say something about the $p$-torsion in $(\Fil_\Tate^*\Pi_{r,q}\sE)_x$.
\end{IntroRems}

The paper is organized as follows: We set the notation in \S \ref{sec:Background}. In \S\ref{sec:SliceTower}  we recall some basic facts about the slice tower, the truncation functors $f_n$ in $\SH(k)$ and $\SH_{S^1}(k)$, and the associated filtration $\Fil_\Tate^*\Pi_{a,b}$.  We recall the construction and basic properties of  the homotopy coniveau tower, a simplicial model for the slice tower in $\SH_{S^1}(k)$, in \S \ref{sec:HCT}. In \S \ref{sec:SpecSeq} we use the simplicial nature of the homotopy coniveau tower to analyze the terms in the slice tower. This leads to the main inductive step in our argument (lemma~\ref{lem:Induction}), and isolates the particular piece that we need to study. This is analyzed further in  \S \ref{sec:Bottom}, where we more precisely identify this piece in terms of a $\sK^{MW}_*$-module structure on the bi-graded homotopy sheaves (see theorem~\ref{thm:GWSliceRevisited}). In \S \ref{sec:MotCoh} we use a decomposition theorem of Morel and results of Cisinski-D\'eglise to prove some boundedness properties of the homotopy sheaves $\Pi_{p,q}\sE$ and their $\Q$-localizations $\Pi_{p,q}\sE_\Q$ for $\sE$ in $\SH_\fin(k)$, under the assumption that the base-field $k$ has finite 2-cohomological dimension.  In the final  section \ref{sec:Finale}, we assemble all the pieces and prove our main result. We conclude with two  appendices; the first   collects some results on norm maps for finite field extensions that are used throughout the paper and the second assembles some basic facts on the localization of compactly generated triangulated categories with respect to a collection of non-zero integers. 

I am grateful to the referee for a number of comments and suggestions for improving an earlier version of this paper. 

\section{Background and notation} \label{sec:Background} Unless we specify otherwise, $k$ will be a fixed perfect base field, without restriction on the characteristic. For details on the following constructions, we refer the reader to \cite{GoerssJardine, Jardine, Jardine2, MorelA1, MorelConn,  MorelLec, MorelVoev}.

We write $[n]$ for the set $\{0,\ldots,n\}$ with the standard order (including $[-1]=\0$) and let $\Delta$ be the category with objects $[n]$, $n=0,1,\ldots$, and morphisms $[n]\to[m]$ the order-preserving maps of sets. Given a category $\sC$, the category of simplicial objects in $\sC$ is as usual the category of functors $\Delta^\op\to\sC$, and the category of cosimplicial objects the functor category $\sC^\Delta$. 

$\Spc$ will denote the category of simplicial sets, $\Spc_\bullet$ the category of pointed simplicial sets, $\sH:=\Spc[WE^{-1}]$ the classical unstable homotopy category and $\sH_\bullet:=\Spc_\bullet[WE^{-1}]$ the pointed version; here $WE$ is the usual class of weak equivalences, that is, morphisms $A\to B$ that induce an isomorphism on all $\pi_n$, for all choice of base-point.  $\Spt$ is the category of spectra, that is, spectrum objects in $\Spc_\bullet$ with respect to the left suspension functor $\Sigma^\ell_{S^1}:=S^1\wedge(-)$. With $sWE$ denoting the class of  stable weak equivalences, that is, morphisms $f:E\to F$ in $\Spt$ that induce an isomorphism on all stable homotopy groups,  $\SH:=\Spt[sWE^{-1}]$ is the classical stable homotopy category. 

For a simplicial object in $\Spc$, resp. $\Spc_\bullet$, resp. $\Spt$, $S:\Delta^\op\to \Spc, \Spc_\bullet, \Spt$, we let $|S|\in\Spc, \Spc_\bullet, \Spt$ denote respective homotopy colimit $\hocolim_{\Delta^\op}S$.

The motivic versions are as follows: $\Sm/k$ is the category of smooth finite type $k$-schemes.  $\Spc(k)$ is the category of $\Spc$-valued presheaves on $\Sm/k$, $\Spc_\bullet(k)$ the $\Spc_\bullet$-valued presheaves, and $\Spt_{S^1}(k)$ the $\Spt$-valued presheaves. These all come with ``motivic''  model structures as simplicial model categories (see for example \cite{Jardine2}); we denote the corresponding homotopy categories by $\sH(k)$, $\sH_\bullet(k)$ and $\SH_{S^1}(k)$, respectively. Sending $X\in \Sm/k$ to the sheaf of sets  on $\Sm/k$ represented by $X$ (which we also denote by  $X$)  gives an embedding of $\Sm/k$ to $\Spc(k)$; we have the similarly defined embedding of the category of smooth pointed schemes over $k$ into $\Spc_\bullet(k)$. Sending a (pointed) simplicial set $A$ to the constant presheaf with value $A$ (also denoted by $A$) defines an embedding of $\Spc$ in $\Spc(k)$ and of $\Spc_\bullet$ in $\Spc_\bullet(k)$. 

Let $\G_m$ be the pointed $k$-scheme $(\A^1\setminus0,1)$. We let $T:=\A^1/(\A^1\setminus\{0\})$ and let $\Spt_T(k)$ denote the category of $T$-spectra, i.e.,  spectra in $\Spc_\bullet(k)$ with respect to the left $T$-suspension functor $\Sigma^\ell_T:= T\wedge(-)$.  $\Spt_T(k)$ has a motivic  model structure (see \cite{Jardine2}) and $\SH(k)$ is the homotopy category. We can also form the category of spectra in $\Spt_{S^1}(k)$ with respect to $\Sigma^\ell_T$;  with an appropriate model structure the resulting homotopy category is equivalent to $\SH(k)$. We will  identify these two homotopy categories without further mention. 

For each $\sA\in \Spc_\bullet(k)$, the suspension functor $\Sigma_\sA:\Spc_\bullet(k)\to \Spc_\bullet(k)$,   $\Sigma_\sA(\sB):=\sB\wedge \sA$, extends to the suspension functor $\Sigma_\sA:\Spt_{S^1}(k)\to 
\Spt_{S^1}(k)$ or  $\Sigma_\sA:\Spt_T(k)\to 
\Spt_T(k)$. For $\sA$ cofibrant, this gives the suspension functors $\Sigma_\sA:\sH_\bullet(k)\to \sH_\bullet(k)$,  $\Sigma_\sA:\SH_{S^1}(k)\to \SH_{S^1}(k)$ and  $\Sigma_\sA:\SH(k)\to \SH(k)$ by applying $\Sigma_\sA$ to a cofibrant replacement.

Both $\SH_{S^1}(k)$ and $\SH(k)$ are triangulated categories with suspension functor $\Sigma_{S^1}$. On $\sH_\bullet(k)$, $\SH_{S^1}(k)$ and $\SH(k)$, we have $\Sigma_T\cong \Sigma_{S^1}\circ \Sigma_{\G_m}$; the suspension functors $\Sigma_T$ and $\Sigma_{\G_m}$ on $\SH(k)$ are invertible.  For $\sA\in \Spc_\bullet(k)$, we have an enriched Hom on   $\Spt_{S^1}(k)$ and $\Spt_T(k)$  with values in spectra; we denote the enriched Hom functor by $\sHom(\sA,-)$. This passes to the homotopy categories $\sH_\bullet(k)$,  $\SH_{S^1}(k)$ and $\SH(k)$ to give for $\sA\in \sH_\bullet(k)$ an enriched Hom  $\sHom(\sA,-)$ with values in $\SH$. For $X\in \Sm/k$, $E\in \Spt_{S^1}(k)$, $\sHom(X_+, E)=E(X)$. 

We have the   triangle of  infinite suspension functors $\Sigma^\infty$ and their right adjoints $\Omega^\infty$
\[
\xymatrix{
\sH_\bullet(k)\ar[r]^{\Sigma^\infty_{S^1}}\ar[rd]_{\Sigma^\infty_T}&\SH_{S^1}(k)\ar[d]^{\Sigma^\infty_T}\\
&\SH(k)
}\quad
\xymatrix{
\sH_\bullet(k)&\SH_{S^1}(k)\ar[l]_{\Omega^\infty_{S^1}}\\
&\SH(k)\ar[u]_{\Omega^\infty_T}\ar[ul]^{\Omega^\infty_T}
}
\]
both commutative up to natural isomorphism. These are all left, resp. right derived versions of Quillen adjoint pairs of functors on the underlying model categories. 

For $\sX\in \sH_\bullet(k)$, we have the bi-graded homotopy sheaf $\Pi_{a,b}\sX$, defined for $a, b\ge0$, as the Nisnevich sheaf associated to the presheaf on $\Sm/k$
\[
U\mapsto \Hom_{\sH_\bullet(k)}(\Sigma^{a}_{S^1}\Sigma^b_{\G_m}U_+,\sX);
\]
note the perhaps non-standard indexing. We have the bi-graded homotopy sheaves $\Pi_{a,b}E$ for $E\in\SH_{S^1}(k)$, $b\ge0$, $a\in\Z$, and $\Pi_{a,b}\sE$ for $\sE\in\SH(k)$,  $a,b\in\Z$, by taking the Nisnevich sheaf associated to
\[
U\mapsto \Hom_{\SH_{S^1}(k)}(\Sigma^{a}_{S^1}\Sigma^b_{\G_m}\Sigma^\infty_{S^1}U_+,E)
\text{ or }
U\mapsto \Hom_{\SH(k)}(\Sigma^{a}_{S^1}\Sigma^b_{\G_m}\Sigma^\infty_TU_+,\sE),
\]
as the case may be. We write $\pi_n$ for $\Pi_{n,0}$; for $E\in \Spt_{S^1}(k)$ fibrant, $\pi_nE$ is the Nisnevich sheaf associated to the presheaf $U\mapsto \pi_n(E(U))$.

$\SH(k)$ has the set of compact generators 
\[
\{\Sigma^n_{S^1}\Sigma^m_T\Sigma^\infty_{S^1}X_+, n,m \in\Z,  X\in\Sm/k\}
\]
and 
$\SH_{S^1}(k)$ has the set of compact generators 
\[
\{\Sigma^n_{S^1}\Sigma^m_T\Sigma^\infty_{S^1}X_+, n\in\Z, m\ge0, X\in\Sm/k\}.
\]
For $\SH(k)$, this is \cite[theorem 9.2]{DuggerIsaksen}; the proof of this result goes through without change to yield the statement for $\SH_{S^1}(k)$. As these triangulated categories are both homotopy categories of stable model categories, both admit arbitrary small coproducts. 

For $F$ a finitely generated field extension of $k$, we may view $\Spec F$ as the generic point of some $X\in\Sm/k$ (since $k$ is perfect). Thus, for a Nisnevich sheaf $\sS$ on $\Sm/k$, we may define $\sS(F)$ as the stalk of $\sS$ at $\Spec F\in X$. For an arbitrary field extension $F$  of $k$ (not necessarily finitely generated over $k$), we define $\sS(F)$  as the colimit over $\sS(F_\alpha)$, as $F_\alpha$ runs over subfields of $F$ containing $k$ and finitely generated over $k$.  For a finitely generated field $F$ over $k$, we consider objects such as $\Spec F$, or $\A^n_F$ as pro-objects in $\Spc(k)$ by the usual system of finite-type models; the same holds for related objects such as $\Spec F_+$ in $\sH_\bullet(k)$ or $\Sigma^\infty_{S^1}\Spec F_+$ in $\SH_{S^1}(k)$, etc. We extend this to arbitrary field extensions of $k$ by taking the system of finitely generated subfields. We will usually not explicitly insert the ``pro-'' in the text, but all such objects, as well as morphisms and isomorphisms between them, should be so understood.

\section{Voevodsky's slice tower}\label{sec:SliceTower} 
We begin by recalling definition and basic properties of the Tate-Postnikov tower in $\SH_{S^1}(k)$ and  in $\SH(k)$. We then define the main object of our study: the  filtration on the bi-graded homotopy sheaves of a $T$-spectrum or an $S^1$-spectrum induced by the respective Tate-Postnikov towers.

For $n\ge0$, we let $\Sigma_T^n\SH_{S^1}(k)$ be the localizing subcategory of $\SH_{S^1}(k)$ generated by the (compact) objects $\Sigma^m_T\Sigma^\infty_{S^1}X_+$, with $X\in \Sm/k$ and $m\ge n$. We note that $\Sigma^0_T\SH_{S^1}(k)=\SH_{S^1}(k)$. The inclusion functor $i_n:\Sigma_T^n\SH_{S^1}(k)\to \SH_{S^1}(k)$ admits, by results of Neeman \cite[theorem 4.1]{NeemanGrothDual}, a right adjoint $r_n$; define the functor
$f_n:\SH_{S^1}(k)\to\SH_{S^1}(k)$ by $f_n:=i_n\circ r_n$. The co-unit for the adjunction gives us the natural morphism
\[
\rho_n:f_nE\to E
\]
for $E\in \SH_{S^1}(k)$; similarly, the inclusion $\Sigma_T^m\SH_{S^1}(k)\subset \Sigma_T^n\SH_{S^1}(k)$ for $n<m$ gives the natural transformation  $f_mE\to f_nE$, forming the {\em Tate-Postnikov tower}
\[
\ldots\to f_{n+1}E\to f_nE\to\ldots\to f_0E=E;
\]
we define $f_n:=\id$ for $n<0$. We complete $f_{n+1}E\to f_nE$ to a distinguished triangle
\[
f_{n+1}E\to f_nE\to s_nE\to f_{n+1}E[1];
\]
this distinguished triangle actually characterizes $s_nE$ up to {\em unique} isomorphism, hence this defines a  distinguished triangle that is functorial in $E$. The object $s_nE$ is the {\em $n$th slice} of $E$.

There is an analogous construction in $\SH(k)$: For $n\in\Z$, let 
\[\Sigma^n_T\SH^\eff(k)\subset\SH(k)
\]
be the localizing category generated by the $T$-suspension spectra $\Sigma^m_T\Sigma^\infty_TX_+$, for $X\in\Sm/k$ and $m\ge n$; write $\SH^\eff(k)$ for $\Sigma^0_T\SH^\eff(k)$. As above, the inclusion $i_n:\Sigma^n_T\SH^\eff(k)\to\SH(k)$ admits a right adjoint $r_n$, giving us the truncation functor $f_n$, $n\in\Z$,  and the Tate-Postnikov tower
\[
\ldots\to f_{n+1}\sE\to f_n\sE\to\ldots\to \sE.
\]
We define the layer $s_n\sE$ by a distinguished triangle as above. For   integers $N\ge n$, we let $\rho_{n,N}:f_N\to f_n$ and $\rho_n:f_n\to \id$ denote the canonical natural transformations. We mention the following elementary but useful result.

\begin{lem} \label{lem:SliceCompat} For integers $N, n$, the diagram of natural endomorphisms of $\SH(k)$
\[
\xymatrix{
f_n\circ f_N\ar[d]_{f_n(\rho_{N})}\ar[r]^{\rho_n(f_N)}&f_N\ar[d]^{\rho_N}\\
f_n\ar[r]_{\rho_n}&\id}
\]
commutes. Moreover, for $N\ge n$,  the map $\rho_n(f_N)$ is a natural isomorphism, and for $N\le n$, the map 
$f_n(\rho_{N})$ is a natural isomorphism. The same holds with $\SH_{S^1}(k)$ replacing $\SH(k)$.
\end{lem}

\begin{proof} The first assertion is just the naturality of $\rho_n$ with respect to the morphism $\rho_N:f_N\to\id$. 

Suppose $N\ge n$. Then $\Sigma^N_T\SH^\eff(k)\subset \Sigma^n_T\SH^\eff(k)$ and thus for all $\sE\in \SH(k)$,  $\id:f_N\sE\to f_N\sE$ satisfies the universal property of $\rho_n(f_N\sE):f_n(f_N\sE)\to f_N\sE$, namely, $f_N\sE$ is in $\Sigma^n_T\SH^\eff(k)$ and $\id:f_N\sE\to f_N\sE$ is universal for maps $T\to f_N\sE$ with $T\in  \Sigma^n_T\SH^\eff(k)$. Thus, $\rho_n(f_N\sE)$ is an isomorphism.

If $N\le n$, then for $\sE\in \SH(k)$, $f_n(f_N\sE)$ is in $ \Sigma^n_T\SH^\eff(k)$ and $\rho_n(f_N\sE):f_n(f_N\sE)\to f_N\sE$ is universal for maps $T\to f_N\sE$ with $T\in  \Sigma^n_T\SH^\eff(k)$. Since $ \Sigma^n_T\SH^\eff(k)\subset  \Sigma^N_T\SH^\eff(k)$, the universal property of $\rho_N(\sE):f_N\sE\to \sE$ shows that $\rho_N\circ \rho_n(f_N\sE):f_n(f_N\sE)\to \sE$ is universal for maps $T\to \sE$ with $T\in  \Sigma^n_T\SH^\eff(k)$, and thus $f_n(\rho_N)$ is an isomorphism. The proof for $\SH_{S^1}(k)$ is the same.
\end{proof}

\begin{lem}\label{lem:SliceIso} For $n\in\Z$, there is a natural isomorphism
\begin{equation}\label{eqn:SliceIso}
f_n\Omega^\infty_T\sE\cong \Omega^\infty_Tf_n\sE.
\end{equation}
\end{lem}

\begin{proof}  
First suppose that $n\ge0$.  It follows from \cite[theorem 7.4.1]{LevineHC} that $\Omega^\infty_Tf_n\sE$ is in  $\Sigma_T^n\SH_{S^1}(k)$ and thus we need only show that $\Omega^\infty\rho_n:\Omega^\infty_Tf_n\sE\to \Omega^\infty_T\sE$ satisfies the universal property of $f_n\Omega^\infty_T\sE\to \Omega^\infty_T\sE$. $\Sigma^n_T\SH^\eff(k)$ is generated as a localizing subcategory of $\SH_{S^1}(k)$  by objects $\Sigma^n_TG$, $G\in \SH_{S^1}(k)$, so it suffices to check for objects of this form.  We have
\begin{multline*}
\Hom_{\SH_{S^1}(k)}(\Sigma_T^nG, \Omega^\infty_Tf_n\sE)\cong
\Hom_{\SH(k)}(\Sigma_T^\infty\Sigma_T^nG,  f_n\sE)\\
\cong \Hom_{\SH(k)}(\Sigma_T^n\Sigma_T^\infty G,  f_n\sE)
\xymatrix{\ar[r]_\sim^{\rho_{n*}}&}\Hom_{\SH(k)}(\Sigma_T^n\Sigma_T^\infty G,  \sE)\\
\cong \Hom_{\SH(k)}(\Sigma_T^\infty\Sigma_T^n G,  \sE)\cong \Hom_{\SH_{S^1}(k)}(\Sigma_T^nG, \Omega^\infty_T\sE).
\end{multline*}
It is easy to check that this sequence of isomorphisms is induced by $(\Omega^\infty_T\rho_n)_*$.

Now suppose that $n<0$. Then $f_n\Omega^\infty_T\sE\cong f_0\Omega^\infty_T\sE\cong \Omega^\infty_Tf_0\sE$, so it suffices to show that the map $f_0\sE\to f_n\sE$ induces an isomorphism
$\Omega^\infty_Tf_0\sE\to \Omega^\infty_Tf_n\sE$. But for $F\in \SH_{S^1}(k)$, $\Sigma^\infty_TF$ is in $\SH^\eff(k)$ and 
\begin{multline*}
\Hom_{\SH_{S^1}(k)}(F, \Omega^\infty_Tf_0\sE)\cong
\Hom_{\SH(k)}(\Sigma_T^\infty F,  f_0\sE)\\
\xymatrix{\ar[r]^{\rho_{n,0}}_\sim&} \Hom_{\SH(k)}( \Sigma_T^\infty F,  f_n\sE)\cong\Hom_{\SH_{S^1}(k)}(F, \Omega^\infty_Tf_n\sE).
\end{multline*}
\end{proof}

For $E\in \SH_{S^1}(k)$, we have (by  \cite[theorem 7.4.2]{LevineHC}) the canonical isomorphism
\begin{equation}\label{eqn:SliceIso2}
\Omega^r_{\G_m}f_nE\cong f_{n-r}\Omega^r_{\G_m}E
\end{equation}
for $r\ge0$. As $\Omega_{\G_m}:\SH(k)\to \SH(k)$ is an auto-equivalence,  and restricts to an equivalence
\[
\Omega_{\G_m}:\Sigma^n_T\SH^\eff(k)\to \Sigma^{n-1}_T\SH^\eff(k),
\]
the analogous identity in $\SH(k)$ holds as well, for all $r\in\Z$.

\begin{Def} For $a\in \Z$, $b\ge0$, $E\in \SH_{S^1}(k)$, define the filtration $F^*_\Tate\Pi_{a,b}E$ of $\Pi_{a,b}E$  by
\[
F^n_\Tate\Pi_{a,b}E:=\text{im}(\Pi_{a,b}f_nE\to  \Pi_{a,b}E);\ n\in \Z.
\]
Similarly, for $\sE\in \SH(k)$, $a,b,n\in\Z$, define
\[
F^n_\Tate\Pi_{a,b}\sE:=\text{im}(\Pi_{a,b}f_n\sE\to  \Pi_{a,b}\sE).
\]
\end{Def}
The main object of this paper is to understand $F^n_\Tate\Pi_{a,b}E$ for suitable $E$. For later use, we note the following:

\begin{lem} \label{lem:DegreeShift} 1. For $E\in\SH_{S^1}(k)$, $n,p,a,b\in\Z$ with $p, b,  b-p\ge0$,  the adjunction isomorphism $\Pi_{a,b}E\cong \Pi_{a, b-p}\Omega^p_{\G_m}E$ induces an isomorphism
\[
F^n_\Tate\Pi_{a,b}E\cong F^{n-p}_\Tate\Pi_{a,b-p}\Omega^p_{\G_m}E.
\]
Similarly, for $\sE\in \SH(k)$, $n,p,a,b\in \Z$,  the adjunction isomorphism $\Pi_{a,b}\sE\cong \Pi_{a.b-p}\Omega^p_{\G_m}\sE$ induces an isomorphism
\[
F^n_\Tate\Pi_{a,b}\sE\cong F^{n-p}_\Tate\Pi_{a,b-p}\Omega^p_{\G_m}\sE.
\]
2. For $\sE\in \SH(k)$, $a,b,n\in\Z$, with $b\ge0$, we have a canonical isomorphism
\[
\phi_{\sE,a,b,n}: \Pi_{a,b}f_n\sE\to \Pi_{a,b}f_n\Omega^\infty_T\sE,
\]
inducing an isomorphism $F^n_\Tate\Pi_{a,b}\sE\cong F^n_\Tate\Pi_{a,b}\Omega_T^\infty\sE$.
\end{lem}

\begin{proof} (1) By \eqref{eqn:SliceIso2}, adjunction induces isomorphisms
\begin{multline*}
F^n_\Tate\Pi_{a,b}E:=\text{im}(\Pi_{a,b}f_nE\to \Pi_{a,b}E)\cong \text{im}(\Pi_{a,b-p}\Omega^p_{\G_m}f_nE\to \Pi_{a,b-p}\Omega^p_{\G_m}E)\\
=  \text{im}(\Pi_{a,b-p} f_{n-p}\Omega^p_{\G_m}E\to \Pi_{a,b-p}\Omega^p_{\G_m}E)
=F^{n-p}_\Tate\Pi_{a,b-p}\Omega^p_{\G_m}E.
\end{multline*}
The proof for $\sE\in \SH(k)$ is the same.

For (2), the isomorphism $\phi_{\sE,a,b,n}$ arises from \eqref{eqn:SliceIso} and the adjunction isomorphism
\begin{align*}
\Hom_{\SH_{S^1}(k)}(\Sigma^{a}_{S^1}\Sigma^b_{\G_m}\Sigma^\infty_{S^1}U_+, f_n\Omega^\infty_T\sE)&\cong
\Hom_{\SH_{S^1}(k)}(\Sigma^{a}_{S^1}\Sigma^b_{\G_m}\Sigma^\infty_{S^1}U_+, \Omega^\infty_Tf_n\sE)\\
&\cong
\Hom_{\SH(k)}(\Sigma^{a}_{S^1}\Sigma^b_{\G_m}\Sigma^\infty_TU_+, f_n\sE).
\end{align*}
\end{proof}

\section{The homotopy coniveau tower}\label{sec:HCT}  Our computations rely heavily on our model for the Tate-Postnikov tower in $\SH_{S^1}(k)$, which we briefly recall (for details, we refer the reader to \cite{LevineHC}). 

We start with the cosimplicial scheme $n\mapsto \Delta^n$, with $\Delta^n$ the {\em algebraic $n$-simplex} $\Spec k[t_0,\ldots, t_n]/\sum_it_i-1$. The cosimplicial structure is given by sending a map $g:[n]\to [m]$ to the map $\Delta(g):\Delta^n\to \Delta^m$ determined by
 \[
 \Delta(g)^*(t_i)=\begin{cases}\sum_{j, g(j)=i}t_j&\text{ if }g^{-1}(i)\neq\0\\0&\text{ else.}\end{cases}
 \]
 A {\em face} of $\Delta^m$ is a closed subscheme $F$ defined by equations $t_{i_1}=\ldots=t_{i_r}=0$; we let $\partial\Delta^n\subset \Delta^n$ be the closed subscheme defined by $\prod_{i=0}^nt_i=0$, i.e.,  $\partial\Delta^n$ is the union of all the proper faces.

 Take $X\in \Sm/k$. We let $\sS_X^{(q)}(m)$ denote the set of closed subsets $W\subset X\times\Delta^m$ such that 
 \[
 \codim_{X\times F}W\cap X\times F\ge q
 \]
 for all faces $F\subset \Delta^m$ (including $F=\Delta^m$). We make $\sS_X^{(q)}(m)$ into a partially ordered set via inclusions of closed subsets. Sending $m$ to  $\sS_X^{(q)}(m)$ and $g:[n]\to [m]$ to $\Delta(g)^{-1}:\sS_X^{(q)}(m)\to \sS_X^{(q)}(n)$ gives us the simplicial poset $\sS_X^{(q)}$.

 Now take $E\in \Spt_{S^1}(k)$. For $X\in\Sm/k$ and closed subset $W\subset X$, we have the spectrum with supports $E^W(X)$ defined as the homotopy fiber of the restriction map $E(X)\to E(X\setminus W)$. This construction is functorial in the pair $(X,W)$, where we define a map $f:(Y,T)\to (X,W)$ as a morphism $f:Y\to X$ in $\Sm/k$ with $f^{-1}(W)\subset T$. We usually denote the map induced by $f:(Y,T)\to (X,W)$ by $f^*:E^W(X)\to E^T(Y)$, but for $f=\id_X:(X,T)\to (X,W)$, $i:W\to T$ the resulting inclusion, we write $i_*:E^W(X)\to E^T(X)$ for $\id_X^*$.

 Define
 \[
 E^{(q)}(X,m):=\hocolim_{W\in \sS_X^{(q)}(m)}E^W(X\times\Delta^m).
 \]
 The fact that $m\mapsto \sS_X^{(q)}(m)$ is a simplicial poset, and $(Y,T)\mapsto E^T(Y)$ is a functor from the category of pairs to spectra shows that $m\mapsto  E^{(q)}(X,m)$ defines a simplicial spectrum. We define the spectrum $E^{(q)}(X)$ by
  \[
 E^{(q)}(X):=|m\mapsto  E^{(q)}(X,m)|:= \hocolim_{\Delta^\op}  E^{(q)}(X,-).
 \]
 
 For $q\ge q'$, the inclusions $\sS_X^{(q)}(m)\subset \sS_X^{(q')}(m)$ induce a map of simplicial posets $\sS_X^{(q)}\subset \sS_X^{(q')}$ and thus a morphism of spectra $i_{q',q}:E^{(q)}(X)\to E^{(q')}(X)$.  Since $E^{(0)}(X,0)=E(X)$,  we have the canonical map
 \[
 \epsilon_X:E(X)\to  E^{(0)}(X),
 \]
 which is a weak equivalence if $E$ is homotopy invariant. Together, this forms the {\em augmented homotopy coniveau tower} 
 \[
E^{(*)}(X):=\ldots\to E^{(q+1)}(X)\xrightarrow{i_{q}}E^{(q)}(X)\xrightarrow{i_{q-1}}\ldots E^{(1)}(X)
 \xrightarrow{i_{0}}E^{(0)}(X)\xleftarrow{\epsilon_X}E(X)
\]
with $i_q:=i_{q,q+1}$. Thus, for homotopy invariant $E$, we have the
homotopy coniveau tower in $\SH$
 \[
E^{(*)}(X):= \ldots\to E^{(q+1)}(X)\xrightarrow{i_{q}}E^{(q)}(X)\xrightarrow{i_{q-1}}\ldots E^{(1)}(X)
 \xrightarrow{i_{0}}E^{(0)}(X)\cong E(X).
\]

Letting $\Sm\ds k$ denote the subcategory of $\Sm/k$ with the same objects and with morphisms the smooth morphisms, it is not hard to see that sending $X$ to $E^{(*)}(X)$ defines a functor from  $\Sm\ds k^\op$ to augmented towers of spectra.

 On the other hand, for $E\in \Spt_{S^1}(k)$, we have the (augmented) Tate-Postnikov tower
 \[
 f_*E:=\ldots\to f_{q+1}E\to f_qE\to\ldots\to f_0E\cong E
 \]
 in $\SH_{S^1}(k)$, which we may evaluate at $X\in\Sm/k$, giving the  tower $f_*E(X)$ in  $\SH$, augmented over $E(X)$.

Call $E\in \Spt_{S^1}(k)$ {\em quasi-fibrant} if, for $E\to E^{fib}$ a fibrant replacement in the motivic model structure, the map $E(X)\to E^{fib}(X)$
is a stable weak equivalence in $\Spt$ for all $X\in \Sm/k$.  As a general rule, we will represent an $E\in \SH_{S^1}(k)$ by a fibrant object in $\Spt_{S^1}(k)$, also denoted $E$, without making explicit mention of this choice.

As a direct consequence of  \cite[theorem 7.1.1]{LevineHC} we have
 \begin{thm} \label{thm:HC} Let $E$ be a quasi-fibrant object in $\Spt_{S^1}(k)$, and take $X\in \Sm/k$. Then there is an isomorphism of augmented towers in $\SH$
 \[
( f_*E)(X)\cong E^{(*)}(X)
\]
over the identity on $E(X)$, which is natural with respect to smooth morphisms in $\Sm/k$.
 \end{thm}

  In particular, we may use the model $E^{(q)}(X)$ to understand $(f_qE)(X)$.

\begin{rem} For $X, Y\in \Sm/k$ with given $k$-points $x\in X(k)$, $y\in Y(k)$, we have a natural isomorphism in $\SH_{S^1}(k)$
\[
\Sigma^\infty_{S^1}(X\wedge Y)\oplus \Sigma^\infty_{S^1}(X\vee Y)\cong \Sigma^\infty_{S^1}(X\times Y),
\]
using the additivity of the category  $\SH_{S^1}(k)$. Thus, $\Sigma^\infty_{S^1}(X\wedge Y)$ is a canonically defined summand of $\Sigma^\infty_{S^1}(X\times Y)$. In particular for $E$ a quasi-fibrant object of $\Spt_{S^1}(k)$, we have a natural isomorphism in $\SH$
\[
\sHom(X\wedge Y,E)\cong \hofib\left(E(X\times Y)\to \hofib(E(X)\oplus E(Y)\to E(k))\right)
\]
where the maps are induced by the evident restriction maps. In particular, we may define $E(X\wedge Y)$ via the above isomorphism, and our comparison results for Tate-Postnikov tower and homotopy coniveau tower extend to values at smash products of smooth pointed schemes over $k$.
\end{rem}

\section{The simplicial filtration}\label{sec:SpecSeq}
In this section, we study the filtration on  $\pi_rf_nE(X)$ induced by the simplicial structure of the model $E^{(n)}(X)$.

\begin{lem}\label{lem:Homotopy1} Let $S$ be a smooth $k$-scheme, $W\subset S\times\A^1$ a closed subset such that $p:W\to S$ is finite. Let $E\in\Spt_{S^1}(k)$ be quasi-fibrant. Then the map induced by the inclusion $i:W\to p^{-1}(p(W))$ induces the zero map 
\[
i_*:\pi_*(E^W(S\times\A^1))\to \pi_*(E^{p^{-1}(p(W))}(S\times\A^1)).
\]
\end{lem}
\begin{proof} We steal a proof of Morel:  Let $Z=p(W)$, and let $j_0:S\times\A^1\to S\times\P^1$ be the standard open neighborhood of $S\times0$ in $S\times\P^1$.  Since $W$ is finite over $S$, $W$ is closed in $S\times\P^1$, so we have the following commutative diagram
\begin{equation}\label{eqn:Eqn1}
\xymatrix{
\pi_r(E^W(S\times\P^1))\ar[r]^{\bar{i}_*}\ar[d]_{j_0^*}& \pi_r(E^{Z\times\P^1}(S\times\P^1))\ar[d]^{j_0^*}\\
\pi_r(E^W(S\times\A^1))\ar[r]_{i_*}& \pi_r(E^{Z\times\A^1}(S\times\A^1)),
}
\end{equation}
where $\bar{i}:W\to Z\times\P^1$ is the inclusion. Let $i_\infty:S\to S\times\P^1$ be the infinity section. Since $W\cap S\times\infty=\0$, the composition
\[
\pi_r(E^W(S\times\P^1))\xrightarrow{\bar{i}_*}\pi_r(E^{Z\times\P^1}(S\times\P^1))\xrightarrow{i_\infty^*}\pi_r(E^{Z\times\infty}(S\times\infty))
\]
is the zero map. Letting $j_\infty:S\times\A^1\to S\times\P^1$ be the standard open neighborhood of $S\times\infty$ in $S\times\P^1$, the restriction map
\[
i_\infty^*:\pi_r(E^{Z\times\A^1}(S\times\A^1))\to \pi_r(E^{Z\times\infty}(S\times\infty))
\]
is an isomorphism, hence 
\[
j_\infty^*\circ\bar{i}_*:\pi_r(E^W(S\times\P^1))\to  \pi_r(E^{Z\times\A^1}(S\times\A^1))
\]
is the zero map. Write $\tilde{j}_\infty$ for the inclusions of  $S\times \P^1\setminus\{0,\infty\}$ into $j_0(S\times \A^1)$ and $\tilde{j}_0$ for the inclusions of  $S\times \P^1\setminus\{0,\infty\}$ into $j_\infty(S\times \A^1)$. Combining \eqref{eqn:Eqn1} with the commutativity of the diagram
\[
\xymatrix{
\pi_r(E^W(S\times\P^1))\ar[r]^{j_\infty^*\circ\bar{i}_*}\ar[d]_{j_0^*}& \pi_r(E^{Z\times\A^1}(S\times\A^1))\ar[d]^{\tilde{j}_0^*}\\
\pi_r(E^W(S\times\A^1))\ar[r]_-{\tilde{j}_\infty^*\circ i_*}& \pi_r(E^{Z\times\P^1\setminus\{0,\infty\}}(S\times\P^1\setminus\{0,\infty\}))
}
\]
we see that  $\tilde{j}_\infty^*\circ i_*=0$.  From the long exact localization sequence
\begin{multline*}
\ldots\to  \pi_r(E^{Z\times0}(S\times\A^1)) \xrightarrow{i_{0*}} \pi_r(E^{Z\times\A^1}(S\times\A^1))\\\xrightarrow{\tilde{j}_\infty^*}\pi_r(E^{Z\times\A^1\setminus\{0\}}(S\times\A^1\setminus\{0\}))\to\ldots
\end{multline*}
we see that
\[
 i_*(\pi_r(E^W(S\times\A^1)))\subset i_{0*}( \pi_r(E^{Z\times0}(S\times\A^1)))\subset \pi_r(E^{Z\times\A^1}(S\times\A^1)).
 \]
 But  $i_{0*}:\pi_r(E^{Z\times0}(S\times\A^1))\to \pi_r(E^{Z\times\A^1}(S\times\A^1))$
 is the zero map, since $i_1^*:\pi_r(E^{Z\times\A^1}(S\times\A^1))\to \pi_r(E^{Z\times1}(S\times1))$ is an isomorphism and $i_1^*\circ i_{0*}=0$.
 \end{proof}

\begin{lem} \label{lem:Homotopy2} Suppose $F$ is   infinite. Take $W\in \sS_F^{(n)}(p)$ and suppose $\codim_{\Delta^p_F}(W)>n$. Then the canonical map
$E^W(\Delta^p_F)\to E^{(n)}(\Spec F,p)$ induces the zero map on $\pi_*$. \end{lem}

\begin{proof}  We identify $\Delta^p$ with $\A^p$  via the barycentric coordinates $t_1,\ldots, t_p$. Suppose $W$ has dimension $d<p-n$. Then $d\le p-1$ and, as $F$ is infinite,  a general linear projection $L:\A^p\to \A^{p-1}$ restricts to $W$ to a finite morphism $W\to \A^{p-1}$. In addition,  $W':=L^{-1}(L(W))$ is in $\sS_F^{(n)}(p)$ for $L$ suitably general. Letting $i:W\to W'$ be the inclusion, it suffices to show that the map
\[
i_*:\pi_*E^W(\Delta^p_F)\to \pi_* E^{W'}(\Delta^p_F)
\]
is the zero map. Via an affine linear change of coordinates on $\Delta^p$, we may identify $\Delta^p$ with $\A^{p-1}\times\A^1$ and  $L:\A^p\to \A^{p-1}$ with the projection $\A^{p-1}\times\A^1\to \A^{p-1}$. The result thus follows from lemma~\ref{lem:Homotopy1}.
\end{proof}

Let $(\Delta_F^p,\del\Delta^p)^{(n)}$ be the set of codimension $n$ points $w$ of $\Delta^p_F$ such that $\overline{\{w\}}$ is in $\sS_F^{(n)}(p)$.

\begin{lem}\label{lem:generic}
Let $F$ be an infinite field. Then the restriction maps
\[
E^W(\Delta^p_F)\to \oplus_{w\in (\Delta^p_F,\del\Delta^p)^{(n)}\cap W}E^w(\Spec\sO_{\Delta^p_F,w})
\]
for $W\in \sS_F^{(n)}(p)$ defines an injection
\[
\pi_r(E^{(n)}(F, p))\to \oplus_{w\in (\Delta^p_F,\del\Delta^p)^{(n)}}\pi_rE^w(\Spec\sO_{\Delta^p_F,w})
\]
for each $r\in\Z$.
\end{lem}

\begin{proof} Take $W\in \sS_F^{(n)}(p)$. Since $\Delta^p_F$ is affine, we can find a $W'\in \sS_F^{(n)}(p)$ of pure codimension $n$ with $W'\supset W$: just take a sufficiently general collection of $n$ functions $f_1,\ldots, f_n$ vanishing on $W$ and let $W'$ be the common zero locus of the $f_i$. Thus the set of pure codimension $n$ subsets $W'$  of $\Delta^p_F$ with $W'\in \sS_F^{(n)}(p)$ is cofinal in  $\sS_F^{(n)}(p)$. 

Let $W\in  \sS_F^{(n)}(p)$ have pure codimension $n$ on $\Delta^p_F$ and let $W_0\subset W$ be any closed subset. Then $W_0$ is also in 
$\sS_F^{(n)}(p)$ and we have the long exact localization sequence
\[
\ldots\to \pi_rE^{W_0}(\Delta^p_F)\xrightarrow{i_{W_0*}} \pi_rE^W(\Delta^p_F)\to \pi_rE^{W\setminus W_0}(\Delta^p_F\setminus W_0)\to\ldots
\]

Let $\sS_F^{(n)}(p)_0\subset  \sS_F^{(n)}(p)$ be the set of all $W_0\in \sS_F^{(n)}(p)$ with $\codim_{\Delta^p_F}W_0>n$. Let
\[
E^{(n)}(F,p)_0=\hocolim_{W_0\in \sS_F^{(n)}(p)_0}E^{W_0}(\Delta^p_F).
\]
Passing to the limit over the above localization sequences gives us the long exact sequence
\[
\ldots\to \pi_rE^{(n)}(F,p)_0\xrightarrow{i_{0*}} \pi_rE^{(n)}(F,p)\to \oplus_{w\in (\Delta_F^p,\del\Delta^p)^{(n)}} \pi_rE^{w}((\Spec\sO_{\Delta^p_F,w}))\to\ldots
\]
By lemma~\ref{lem:Homotopy2}, the  map $i_{0*}$ is the zero map, which proves the lemma.
\end{proof}

Let $S:\Delta^\op\to \Spt$ be a simplicial spectrum, $|S|=\hocolim_{\Delta^\op}S\in \Spt$ the associated spectrum, giving us the spectral sequence
\[
E^1_{p,q}=\pi_qS(p)\Longrightarrow \pi_{p+q}|S|.
\]
This spectral sequence induces an increasing filtration $\Fil^{simp}_*\pi_r|S|$ on $\pi_r|S|$. We have the $q$-truncated simplicial spectrum $S_{\le q}$ and $\Fil^{simp}_q\pi_r|S|$ is just the image of $\pi_r|S_{\le q}|$ in $\pi_r|S|$. In particular $\Fil^{simp}_{-1}\pi_r|S|=0$ and $\cup_{q=0}^\infty\Fil^{simp}_q\pi_r|S|=\pi_rS$, so the spectral sequence is weakly convergent, and is strongly convergent if for instance there is an integer $q_0$ such that $S(p)$ is $q_0$-connected for all $p$.

The isomorphism of theorem~\ref{thm:HC} thus gives us the weakly convergent spectral sequence
\begin{equation}\label{eqn:SimpSpecSeq}
E^1_{p,q}(X,E,n)= \pi_q E^{(n)}(X,p)\Longrightarrow\pi_{p+q}f_nE(X)
\end{equation}
which is strongly convergent if   $\pi_qE^{(n)}(X,p)=0$ for $q\le q_0$, independent of $p$.  This defines the increasing filtration $\Fil^{simp}_*(E)\pi_rf_nE(X)$ of  $\pi_rf_nE(X)$ with associated graded $\gr^{simp}_p(E)\pi_rf_nE(X)=E^\infty_{p,r-p}$.

\begin{lem} \label{lem:Structure} Suppose that $k$ is infinite and  that $\Pi_{a,*}E(K)=0$ for $a<0$ and all fields $K$ over $k$.  Let $F\supset k$ be a field extension of $k$. Then
\begin{enumerate}
\item $E^1_{p,r-p}(F, E,n)=0$ for  $p> r+n$ and $\Fil^{simp}_{r+n}(E)\pi_rf_nE(F)=
\pi_rf_nE(F)$.
\item $E^1_{p,q}(F, E,n)$ is isomorphic to a subgroup of 
$\oplus_{w\in (\Delta^p_F,\del\Delta^p)^{(n)}}\Pi_{q+n,n}E(F(w))$.
\item The spectral sequence \eqref{eqn:SimpSpecSeq} is strongly convergent.
\end{enumerate}
\end{lem}

\begin{proof} Since the spectral sequence is weakly convergent, to prove (3) it suffices to show that $E^1_{p,q}(F, E,n)=0$ for $q<-n$ and for 
 (1) it suffices to show that $E^1_{p,r-p}=0$ for $p> r+n$. These both follows from (2)  as our hypothesis implies that $\Pi_{a,*}E(F(w))=0$ for $a<0$, $w\in \Delta^p_F$. 

For (2), lemma~\ref{lem:generic} gives us an inclusion
\[
E^1_{p,q}=\pi_{q}E^{(n)}(F,p)\subset  \oplus_{w\in (\Delta^p_F,\del\Delta^p)^{(n)}}\pi_{q}E^w(\Spec\sO_{\Delta^p_F,w}).
\]
Take $w\in (\Delta^p_F,\del\Delta^p)^{(n)}$. By the Morel-Voevodsky purity isomorphism \cite[\hbox{\it loc.cit.}]{MorelVoev}, we have $E^w(\Spec\sO_{\Delta^p_F,w})\cong \sHom(\Sigma^n_{S^1}\Sigma^n_{\G_m} w_+,E)$, hence
\[
\pi_{q}E^w(\Spec\sO_{\Delta^p_F,w})\cong \Pi_{q+n,n}E(F(w)),
\]
which proves (2).
\end{proof}

For a field extension $K$ of $k$, we write $\trdim_kK$ for the transcendence dimension of $K$ over $k$. 

\begin{lem}\label{lem:Induction} Let $E$ be in $\SH_{S^1}(k)$ and suppose  $\Pi_{a,*}E(K)=0$ for $a<0$ and all fields $K$ over $k$. Let $p$ be the exponential characteristic of $k$.  Let $r$ be an integer. Suppose we have functions 
\[
(d,q)\mapsto N_j(d,q; E)\ge0;\  d, q\ge0, j=0,\ldots, r-1,
\]
such that, for each  field extension  $K$ of $k$ with $\trdim_kK\le d$, each $j=0,\ldots, r-1$,   and all integers  $q, M\ge0$, $m\ge N_j(d, q; E)$, 
we have
\begin{equation}\label{eqn:IndAssump}
F^m_\Tate \Pi_{j,q}f_ME(K)[1/p]=0.  
\end{equation}
Let $F$ be a field extension of $k$ with $\trdim_kF\le d$ and fix an integer $n\ge0$. For $r>0$, let $N=\max_{j=0}^{r-1}N_j(r-j+d,n; E)$; for $r\le 0$, set $N=0$. Then for  all integers $m\ge N$,  $n\ge0$,  , we have
\[
\pi_r(f_n(f_mE))(F))[1/p]= \Fil^{simp}_n(f_mE)\pi_r(f_n(f_mE))(F)[1/p].
\]
\end{lem}

\begin{proof} We delete the ``$[1/p]$" from the notation in the proof, using the convention that we have inverted the exponential characteristic $p$ throughout.  

If $k$ is a finite field, fix a prime $\ell$ and let $k_\ell$ be the union of all $\ell$-power extensions of $k$. If we know the result for $k_\ell$, then using  proposition~\ref{prop:InvertP}(2)  for each $k\subset k'\subset k_\ell$ with $k'$ finite over $k$ proves the result for $k$, after inverting $\ell$. Doing the same for some $\ell'\neq \ell$, we reduce to the case of an infinite field $k$.

By \cite[proposition 3.2]{LevineGW}, the hypothesis  $\Pi_{a,*}E(K)=0$ for $a<0$ and all $K$  implies $\Pi_{a,*}f_m E(K)=0$ for $a<0$, all $K$,  and all $m\ge0$. In particular, $\pi_r(f_n(f_mE))(F))=0$ for $r<0$ and all $n,m\ge0$, so for $r<0$, the lemma is trivially true. We therefore assume $r\ge0$. In addition, it  follows from lemma~\ref{lem:Structure}(1) that  the spectral  sequence \eqref{eqn:SimpSpecSeq},  $E^*_{*,*}(F, f_mE,n)$, is strongly convergent for all $m, n\ge0$ and $E^1_{p, r-p}(F, f_mE,n)=0$ for $p>r+n$. Since $\Fil^{simp}_*(f_m E)\pi_r(f_n(f_mE))(F)$ is by definition the filtration on $\pi_r(f_n(f_mE))(F)$ induced by this spectral sequence, we need only show that
\[
E^1_{p, r-p}(F, f_mE, n)=0\text{ for }n<p\le r+n\text{ and } m\ge N.
\]
In particular, the result is proved for $r=0$; we now assume $r>0$.

Let $p$ be an integer, $n<p \le r+n$. By lemma~\ref{lem:Structure}(2), 
\[
E^1_{p,r-p}(F, f_mE,n)\subset \oplus_{w\in (\Delta^p_F,\del\Delta^p)^{(n)}}\Pi_{r-p+n,n}f_mE(k(w)).
\]
For $w\in (\Delta^p_F,\del\Delta^p)^{(n)}$,  $w$ has codimension $n$ on $\Delta^p_F$, hence $\trdim_Fk(w)=p-n$  and thus   $\trdim_kk(w)\le p-n+d$. We have  $m\ge N_{r-p+n}(p-n+d, n; E)$ since $0\le r-p+n<r$, and so our hypothesis \eqref{eqn:IndAssump} implies
\[
F^m_\Tate \Pi_{r-p+n,n}f_mE(k(w))=0.
\]
But $F^m_\Tate \Pi_{a, b}f_mE(k(w))=\im(\Pi_{a,b}f_mf_mE(k(w))\xrightarrow{\rho_m(f_mE)_*} \Pi_{a,b}f_mE(k(w)))$ by definition,  and  $\rho_m(f_mE)$ is an isomorphism (lemma~\ref{lem:SliceCompat}), hence 
\[
F^m_\Tate \Pi_{a,b}f_mE(k(w))=\Pi_{a,b}f_mE(k(w)).
\]
Thus $\Pi_{r-p+n,n}f_mE(k(w))=0$ and hence $E^1_{p,r-p}=0$ for $n<p \le r+n$, as desired. 
\end{proof}

\section{The bottom of the filtration}\label{sec:Bottom}
In this section, $k$ will be a fixed perfect base field. We study the subgroup $\Fil^{simp}_n(E)\pi_rf_nE(F)$  isolated in lemma~\ref{lem:Induction}.

\begin{lem}\label{lem:TrivVan} Let $E$ be in $\SH_{S^1}(k)$.  Then
$\Fil^{simp}_{n-1}(E)\pi_rf_nE(F)=0$ for all fields $F$ over $k$.
\end{lem}

\begin{proof}  For any $X\in\Sm/k$, $\Fil^{simp}_q\pi_rE^{(n)}(X)$ is by definition the image in $\pi_rE^{(n)}(X)$ of $\pi_r|E^{(n)}(X,-\le q)|$, 
where $E^{(n)}(X,-\le q)$ is the $q$-truncated simplicial spectrum associated to $E^{(n)}(X,-)$. For $X=\Spec F$, we clearly have $\sS_F^{(n)}(p)=\0$ for $p<n$, as $\Delta^p_F$ has no closed subsets of codimension $>p$. Thus $|E^{(n)}(X,-\le q)|$ is the 0-spectrum for $q<n$ and hence $\Fil^{simp}_{n-1}\pi_rE^{(n)}(F)=0$.
\end{proof}

To study the first non-zero layer $\Fil^{simp}_n(E)\pi_rf_nE(F)$ in $\Fil^{simp}_*(E)\pi_rf_nE(F)$, we apply the results of \cite{LevineGW}. For this, we  recall some of these results and constructions.

We let $V_n=(\Delta^1_F\setminus\del\Delta^1)^n$. The function $-t_1/t_0$ on $\Delta^1$ gives an open immersion $\rho_n:V_n\to \A^n$, identifying $V_n$ with $(\A^1\setminus\{0,1\})^n$. 

Suppose that $E$ is an $n$-fold $T$-loop spectrum, that is, there is an object $\omega_T^{-n}E\in\Spt(k)$ and an isomorphism $E\cong \Omega_T^n\omega_T^{-n}E$ in $\SH_{S^1}(k)$. Given an $n$-fold delooping $\omega_T^{-n}E$ of $E$ , we have explained in \cite[\S 5]{LevineGW} how to construct a ``transfer map''
\[
\Tr_{F(w)/F}^*:\pi_*E(w)\to \pi_*E(F),
\]
for each closed point $w\in \A^n_{F}$, separable over $F$.

If now $E=\Omega^\infty_T\sE$ for some $T$-spectrum $\sE\in\SH(k)$, then the bi-graded homotopy sheaves $\Pi_{*,*}E$ admit a canonical right action by the bi-graded homotopy sheaves of the sphere spectrum $\mS_k\in\SH(k)$:
\[
\Pi_{a,b}E\otimes\Pi_{p,q}(\mS_k)\to \Pi_{a+p, b+q}E
\]

Hopkins and Morel \cite[theorem 6.3.3]{MorelLec} have defined a graded ring homomorphism 
\begin{equation}\label{eqn:HopkinsMorel}
\theta_*(F):\oplus_{n\in\Z}K^{MW}_n(F)\to \oplus_{n\in\Z}\Pi_{0,-n}\mS_k(F),
\end{equation}
natural in the field extension $F$ of $k$.\footnote{A result of Morel \cite[corollary 6.41]{MorelA1} implies that  $\theta_*(F)$ is an isomorphism for all fields $F$,  but we will not need this.} Via $\theta_*(F)$, the right $\Pi_{p,q}(\mS_k)$-action gives a right action, natural in $F$,
\[
\Pi_{a,b}E(F)\otimes K^{MW}_*(F) \to \Pi_{a, b-*}E(F).
\]
This gives us the filtration $F_{MW}^n\Pi_{a,b}E(F)$ of $\Pi_{a,b}E$, defined by
\[
F_{MW}^n\Pi_{a,b}E(F):=\text{im}[\Pi_{a,b+n}E(F)\otimes K^{MW}_{n}(F)\to \Pi_{a,b}E(F)];\quad n\ge 0.
\]
Completing with respect to the transfer maps gives us the filtration  $F^{n}_{MW^{Tr}}\Pi_{a,b}E(F)$:
\begin{Def}[see \hbox{\cite[definition 7.9]{LevineGW}}] Let $E=\Omega_T^\infty\sE$ for some $\sE\in \SH(k)$, $F$ a field extension of $k$. Take integers $a, b, n$ with $n, b\ge0$.  Let  $F^{n}_{MW^{Tr}}\Pi_{a,b}E(F)$ denote the subgroup of $\Pi_{a,b}E(F)$ generated by elements of the form
\[
Tr_F(w)^*(x);\quad x\in F_{MW}^{n}\Pi_{a,b}E(F(w))
\]
as $w$ runs over closed points of $\A^n_F$, separable over $F$. We write $F^*_{MW^{Tr}}\pi_rE(F)$ for the filtration $F^*_{MW^{Tr}}\Pi_{r,0}E(F)$ on $\pi_rE(F)=\Pi_{r,0}E(F)$
\end{Def}

Theorem 7.11 of \cite{LevineGW} expresses the ``Tate-Postnikov'' filtration $F_\Tate^*\pi_{0}E(F)$ on $\pi_{0}E(F)$ in terms of the ``Milnor-Witt'' filtration $F_{MW}^*\pi_0E(F)$, under the connectivity assumption $\Pi_{a,*}E=0$ for $a<0$. With some minor changes,  the proof of  this result  goes through to show:

\begin{thm}\label{thm:GWSliceRevisited} Let  $\sE$ be in $\SH(k)$ and let $E=\Omega^\infty_T\sE$. Suppose that $\Pi_{a,*}E=0$ for $a<0$. Let $F$ be a perfect field extension of $k$. Then for all $r\ge0$, $n\ge0$,  we have an equality of subgroups of $\Pi_{r,q}E(F)$:
\[
\rho_n(E)(\Fil^{simp}_n(E)\pi_rf_nE(F))=F^n_{MW^{Tr}}\pi_{r}E(F)
\]
\end{thm}

\begin{proof}[Sketch of proof.] We briefly indicate the changes that need to be made in the arguments for theorem 7.11 {\it loc.\,cit.}: In \cite[prop. 4.3 and thm. 7.6]{LevineGW}, replace $\pi_0$ with $\pi_r$ and $F^n_\Tate\pi_0(E)(F)$ with $\rho_n(E)(\Fil^{simp}_n(E)\pi_{r}f_nE(F))$. Also, instead of using the fact that $\theta_*(F)$ is an isomorphism, we need only use that   $\theta_*(F)$ is a ring homomorphism, and that for $u\in F^\times$, $\theta_1([u])$ is the element of $\Pi_{0,-1}\mS_k(F)$ coming from the map $u:\Spec F\to \G_m$ corresponding to $u$ (see \cite[theorem 6.3.3]{MorelLec}).
\end{proof}

\section{Finite spectra and cohomologically finite spectra}\label{sec:MotCoh}
In this section we will have occasion to use some localizations of $\SH(k)$ with respect to multiplicatively closed subsets of $\Z\setminus\{0\}$. To make our discussion precise, we have collected notations and  some elementary  facts  concerning such localizations in Appendix~\ref{sec:Localization}; we will use these without further mention in this section.

For a field $F$, let  $F^{sep}$ denote the separable closure of $F$ and $G_F$ the absolute Galois group $\Gal(F^{sep}/F)$. For $p$ a prime, let $cd_p(G_F)$ denote  the $p$-cohomological dimension of the profinite group $G_F$, as defined in \cite[I, \S 3.1]{Serre}. We write $cd_p(F)$ for $cd_p(G_F)$, and call $cd_p(F)$ the  {\em $p$-cohomological dimension} of $F$.  The cohomological dimension of $F$, $cd(F)$,  is the supremum of the $cd_p(F)$ over all primes $p$.

We turn to our main topic in this section: the study of finite and cohomologically finite objects in $\SH(k)$.

\begin{Def} 1. Let $\SH(k)_\fin(k)$ be the thick subcategory of  $\SH(k)$ generated by objects $\Sigma_T^n\Sigma^\infty_TX_+$ for $X$ a smooth projective $k$-scheme and $n\in \Z$.\\
2. Let $\SH(k)_\cfin(k)$ be the full subcategory of $\SH(k)$ with objects those $\sE$ such that 
\begin{enumerate}
\item[(i)]  there is an integer $d$ such that, for $n>d$,  $\Pi_{r,n}(\sE)_\Q=0$ for all $r$,
\item[(ii)] there is an integer $c$ such that $\Pi_{r, q}\sE=0$ for $r\le c$, $q\in\Z$.
\end{enumerate}
3. For $\sE\in  \SH(k)_\cfin(k)$,  define $d(\sE)$ to be the infimum among integers $d$  such that $\Pi_{r,n}(\sE)_\Q=0$ for all $r$ and for all $n>d$. \\
4. We say that $\sE\in \SH(k)$ is {\em topologically $c$-connected} if $\Pi_{r, *}\sE=0$ for $r\le c$. We let $c(\sE)$ be the supremum among integers $c$  such that  $\sE$ is topologically $c$-connected. \\ 
5. For $E\in \SH_{S^1}(k)$, we say that $E$ is topologically $c$-connected if $\Pi_{r, n}E=0$ for $r\le c$ and $n\ge0$; we say that $E$ is $c$-connected if $\pi_rE=0$ for $r\le c$.
\end{Def}

\begin{rems} 1. $\SH(k)_\cfin(k)$ is a thick subcategory of $\SH(k)$.\\
2. As $\Sigma_T^n\Sigma^\infty_TX_+$  is compact for all $n\in \Z$, $X\in \Sm/k$,  $\SH(k)_\fin(k)$ is contained in the category $\SH(k)^c$ of compact objects in $\SH(k)$. If $k$ admits resolution of singularities, it is not hard to show that $\Sigma_T^n\Sigma^\infty_TX_+$ is in 
$\SH(k)_\fin(k)$ for all $X\in \Sm/k$, $n\in\Z$; as these objects (and their translates) form a set of compact generators for $\SH(k)$, it follows that 
$\SH(k)_\fin(k)=\SH(k)^c$ if $k$ admits resolution of singularities, that is, if $k$ has characteristic zero.
\end{rems}

\begin{rem}\label{rem:ConnectTopConnect} For a presheaf $A$ of abelian groups on $\Sm/k$, we have the presheaf $A_{-1}$ defined by 
\[
A_{-1}(U):=\ker(i_1^*:A(U\times \A^1\setminus\{0\})\to A(U)).
\]

Let $E$ be in $\SH_{S^1}(k)$. By \cite[lemma 4.3.11]{MorelLec}, there is a natural isomorphism $\pi_n(\Omega_{\G_m}E)\cong (\pi_n(E))_{-1}$. In particular, if $E$ is $c$-connected, then so is $\Omega_{\G_m}E$. As $\Pi_{a,b}E=\Pi_{a,b-1}\Omega_{\G_m}E$ and $\Pi_{r,0}=\pi_r$, we see that  $E$ is $c$-connected if and only if $E$ is topologically $c$-connected. 
\end{rem}

\begin{lem}\label{lem:Boundedness} Take  $\sE\in \SH_\fin(k)$. Then there is an integer $n(\sE)$ such that $\rho_n(\sE):f_n(\sE)\to \sE$ is an isomorphism for all $n\le n(\sE)$.
\end{lem}

\begin{proof}  As $f_n$ is exact, it suffices to prove the result for $\sE=\Sigma_T^m\Sigma^\infty_TX_+$ with $X$ smooth and projective over $k$ and $m\in\Z$. As $f_n\circ \Sigma_T^m=\Sigma_T^m\circ f_{n-m}$, we need only prove the result for 
$\sE=\Sigma^\infty_TX_+$.  But  $\Sigma^\infty_TX_+$ is in $\SH^\eff(k)$,  so  $\rho_n(\Sigma^\infty_TX_+)$ is an isomorphism  for all $n\le 0$.\end{proof}

\begin{lem}\label{lem:CohFinVan} Let $\sE$ be in $\SH(k)_\cfin(k)$. \\
1. For  $U\in \Sm/k$, we have $\Hom_{\SH(k)_\Q}((\Sigma^p_{S^1}\Sigma^q_{\G_m}\Sigma^\infty_TU_+)_\Q,\sE_\Q)=0$
for all $p\in\Z$, $q>d(\sE)$.
\\
2. For all $n> d(\sE)$, $(f_n\sE)_\Q\cong 0$ in $\SH(k)_\Q$.
\end{lem}

\begin{proof} Let $i_n: \Sigma^n_T\SH^\eff(k)\to \SH(k)$ be the inclusion, $r_n$ the right adjoint to $i_n$; recall that $f_n:=i_n\circ r_n$. By example~\ref{ex:LocSlice}, the $\Q$-localization $\Sigma^n_T\SH^\eff(k)_\Q$ of $\Sigma^n_T\SH^\eff(k)$ is the localizing subcategory of $\SH(k)_\Q$ generated by the objects $(\Sigma^q_{\G_m}\Sigma^\infty_TU_+)_\Q$, $q\ge n$, $U\in \Sm/k$, the inclusion $\Sigma^n\SH^\eff(k)_\Q\to \SH(k)_\Q$  is given by $i_{n\Q}$, the right adjoint to $i_{n\Q}$ is $r_{n\Q}$ and $f_{n\Q}(\sE_\Q)=(f_n\sE)_\Q$ for all $\sE\in\SH(k)$.

Using this, we see that (1) implies (2), as (1) implies that for $\sF\in\Sigma^n_T\SH^\eff(k)_\Q$, $n>d(\sE)$, 
\[
\Hom_{\SH(k)_\Q}(\sF,\sE_\Q)=0.
\]
Since  $(f_n\sE)_\Q\to \sE_\Q$  is universal for maps $\sF\to \sE_\Q$ with $\sF\in\Sigma^n_T\SH^\eff(k)_\Q$, it follows that $(f_n\sE)_\Q=0$ for $n>d(\sE)$.

For (1), by  lemma~\ref{lem:Localization}(2), we have
\[
\Hom_{\SH(k)_\Q}((\Sigma^p_{S^1}\Sigma^q_{\G_m}\Sigma^\infty_TU_+)_\Q,\sE_\Q)=
\Hom_{\SH(k)}(\Sigma^p_{S^1}\Sigma^q_{\G_m}\Sigma^\infty_TU_+,\sE)_\Q.
\]
Thus, tensoring the usual  Gersten-Quillen spectral sequence with $\Q$ gives the 
strongly convergent spectral sequence
\[
E_1^{a,b}=\oplus_{u\in U^{(a)}}\Pi_{-b,q+a}\sE(k(u)_\Q\Longrightarrow \Hom_{\SH(k)_\Q}((\Sigma^{-a-b}_{S^1}\Sigma^q_{\G_m}\Sigma^\infty_TU_+)_\Q,\sE_\Q),
\]
concentrated in the range $0\le a\le \dim U$. The assumption $q>d(\sE)$ implies that $E_1^{a,b}=0$ for all $a,b$, proving (1).
\end{proof}

\begin{lem}\label{lem:CohFinTate} Let $\sE$ be in $\SH(k)_\cfin(k)$. 
 Then  $f_n\sE$ is  in $\SH(k)_\cfin(k)$, $d(f_n\sE)\le d(\sE)$ and $c(f_n\sE)\ge c(\sE)$, for all $n\in \Z$.
\end{lem}

\begin{proof} We have shown in \cite[proposition 3.2]{LevineGW} that, for $E\in \SH_{S^1}(k)$, if $E$ is topologically $-1$-connected, then $f_pE$ is also topologically $-1$-connected for all $p$. For $E:=\Omega^\infty_T\Sigma^{-c-1}_{S^1}\Sigma^{q}_{\G_m}\sE$, we have
\[
\Pi_{r-c-1, m+q}f_{n+q}E\cong \Pi_{r, m+q}f_{n+q}\Sigma^{q}_{\G_m}\sE
\cong \Pi_{r, m+q}\Sigma^{q}_{\G_m}f_n\sE\cong \Pi_{r, m} f_n\sE
\]
for $m\ge -q$.  Similarly, $\Pi_{r-c-1, m+q}E\cong  \Pi_{r, m} \sE$ for $m\ge-q$. Thus,  if $\Pi_{r,m}\sE=0$ for $r\le c$ and $m\ge -q$,  the same holds for $f_n\sE$.  As $q$ was arbitrary, we see that if $\sE$ is topologically $c$-connected, so is $f_n\sE$ for all $n$, that is, $c(f_n\sE)\ge c(\sE)$.

We have already seen in lemma~\ref{lem:CohFinVan} that $f_n\sE_\Q=0$ for $n> d(\sE)$, hence $f_n\sE$ is in $\SH(k)_\cfin(k)$ and $d(f_n\sE)=-\infty$ for $n>d(\sE)$. For $n\le d(\sE)$, take $m>d(\sE)\ge n$. Then $\Pi_{r,m}(f_n\sE)\cong \Pi_{r,m}(\sE)$, hence $\Pi_{r,m}(f_n\sE)_\Q=0$. Thus $f_n\sE$ is in $\SH(k)_\cfin(k)$  and $d(f_n\sE)\le d(\sE)$.
\end{proof}

In order to proceed further with our study of $\SH_\fin(k)$ and $\SH(k)_\cfin(k)$, we need to recall Morel's decomposition of $\SH(k)[\frac{1}{2}]$. We recall the ring homomorphism \eqref{eqn:HopkinsMorel}
\[
\theta_k:=\theta_0(k):K_0^{MW}(k)\to \End_{\SH(k)}(\mS_k).
\]
 In addition, Morel \cite[lemma 3.10]{MorelA1} has defined an isomorphism of rings $K_0^{MW}(k)\cong \GW(k)$,  with $1+\eta\cdot [u]$ corresponding to the one-dimensional form $ux^2$ if $\Char k\neq2$.

Morel \cite[section 6.1]{MorelLec}  has considered the action of $\Z/2$ on the sphere spectrum $\mS_k$ arising from the exchange of factors
\[
\tau:T\wedge T\to T\wedge T.
\]
We also write $\tau$ for the induced automorphism of the sphere spectrum $\mS_k$. Morel identifies the corresponding element of $\End(\mS_k)$ as
\begin{equation}\label{eqn:Tau}
\tau:=\theta_k(1+\eta\cdot [-1]).
\end{equation}

After inverting 2, the action of $\tau$  decomposes $\mS_k[\frac{1}{2}]$ into its $+1$ and $-1$ eigenspaces
\[
\mS_{k}[{\scriptstyle\frac{1}{2}}]=\mS^+_k[{\scriptstyle\frac{1}{2}}]\oplus \mS^-_k[{\scriptstyle\frac{1}{2}}];
\]
as $\mS_k[\frac{1}{2}]$ is the unit in the tensor category $ \SH(k)[\frac{1}{2}]$, this induces a  decomposition of $\SH(k)[\frac{1}{2}]$ as
\[
 \SH(k)[{\scriptstyle\frac{1}{2}}]=\SH(k)^+\times \SH(k)^-.
\]
This extends to a decomposition of $\SH(k)_\Q$ as $\SH(k)_{\Q}^+\times \SH(k)_{\Q}^-$. For an object $\sE$ of $\SH(k)$, we write the corresponding factors of $\sE_\Q$ as $\sE^+_\Q$, $\sE^-_\Q$. 

\begin{lem}\label{lem:torsion}  Suppose that either
\begin{enumerate}
\item $\Char k=0$ and  $cd_2(k)<\infty$.
\item $\Char k>0$.
\end{enumerate}
Then $2^{N+1}\cdot \eta=0$, where in case (1), $N=\cd_2(k)$, and in case (2), $N=0$ if $\Char k=2$ or if $\Char k=p$ is odd and $p\equiv 1\mod 4$, $N=1$ if $p\equiv 3\mod 4$. Moreover, letting  $I\subset\GW(k)$ be the augmentation ideal, we have $I^n=0$ for $n>\cd_2(k)$, assuming $k$ has characteristic $\neq2$.  
\end{lem}

\begin{proof} Suppose that $\Char k\neq2$.  Sending $(a_1,\ldots, a_n)\in (k^\times)^n$ to the $n$-fold Pfister form $\<\<a_1,\ldots, a_n\>\>$ descends to a well-defined surjective homomorphism $p_n:K^M_n(k)/2\to I^n/I^{n+1}$ \cite[\S 4.1]{Milnor}; in particular $2I^n\subset I^{n+1}$ for all $n\ge0$.    By the Milnor conjecture \cite[theorem 6.6 and theorem 7.4]{VoevMilnor} and \cite[theorem 4.1]{OVV},  $p_n$ is an isomorphism, and induces an isomorphism
\[
I^n/I^{n+1}\cong H^n_\et(k,\mu_2^{\otimes n}).
\]
Thus $I^n=I^{n+1}$ for $n\ge N+1$, where $N:=\cd_2 k$.   By the theorem of Arason-Pfister \cite[Korollar 1]{ArasonPfister}, $\cap_nI^n=\{0\}$, hence $I^{N+1}=0$, and thus $2^{N+1}$ kills the Witt group $W(k)$.

As a $K^{MW}_0(k)$-module, $K_{-1}^{MW}(k)$ is cyclic with generator $\eta$. However, 
\[
K_{-1}^{MW}(k)\cong W(k)
\]
by \cite[lemma 3.10]{MorelA1}, and thus $2^{N+1}\eta=0$. This handles the case (1).

In case $k$ has characteristic $p>0$, then as $\eta$ comes from base extension from the prime field 
$\F_p$, it suffices to show that $2^{N+1}\eta=0$ in $\SH(\F_p)$, with $N$ as in the statement of the lemma.

For $p$ odd, we have $\GW(\F_p)\cong K_0^{MW}(\F_p)$, with $1+\eta[-1]$ corresponding to the form $-x^2$; the hyperbolic form $x^2-y^2$ corresponds to $2+\eta[-1]$. If $p\equiv 1\mod 4$, then $-1$ is a square, hence $-x^2$ and $x^2$ are isometric forms and $1+\eta[-1]=1$ in $K_0^{MW}(\F_p)$. The relation $\eta(2+\eta[-1])=0$ in $K_0^{MW}(\F_p)$ thus simplifies to $2\eta=0$. 

If $p\equiv 3\mod 4$, then $-1$ is a sum of two squares, and hence the quadratic form $x^2+y^2+z^2$ is isotropic. Thus the Pfister form $x^2+y^2+z^2+w^2$ is also isotropic and hence hyperbolic, and hence the form 
 $-x^2-y^2-z^2-w^2$ is hyperbolic as well. Translating this back to $K_0^{MW}(\F_p)$ gives the relation $4(1+\eta[-1])=2(2+\eta[-1])$ or $2\eta[-1]=0$. Combining this with relation $\eta(2+\eta[-1])=0$ yields $4\eta=0$.

In case $k$ has characteristic 2, $-1=+1$. The relation $\eta(2+\eta[-1])=0$ simplifies to $2\eta=0$ (as $[1]=0$ in $K_1^{MW}(F)$ for all fields $F$ \cite[lemma 3.5]{MorelA1}).
\end{proof}

Lemma~\ref{lem:torsion} and \eqref{eqn:Tau}  imply:

\begin{lem} \label{lem:Plus} Suppose that $k$ has finite 2-cohomological dimension or that $\Char k>0$. Then  $\SH(k)[\frac{1}{2}]=\SH(k)^+$.
\end{lem}

\begin{prop}\label{prop:Vanishing1} Take $X\in \Sm/k$. Then\\
1.   $\Pi_{r,n}(\Sigma^\infty_TX_+)=0$ for $r<0$ and   $n\in\Z$\\
2. Suppose $X$ is smooth and projective over $k$ of dimension $d$ and that   $k$ has finite $2$-cohomological dimension or $\Char k>0$. Then $\Pi_{r,n}(\Sigma^\infty_TX_+)_\Q=0$ for $n> d$, $r\in\Z$.\\
3. Suppose  $k$ has finite $2$-cohomological dimension or $\Char k>0$. Then each $\sE$ in $\SH(k)_\fin(k)$ is also in $\SH(k)_\cfin(k)$.
\end{prop}

\begin{proof} Noting that  $\SH(k)_\cfin(k)$ is a thick subcategory of $\SH(k)$, and that 
\[
\Pi_{r,n}(\Sigma^m_T\Sigma^\infty_TX_+)=\Pi_{r-m,n-m}(\Sigma^\infty_TX_+),
\] 
we see that (3) follows from (1) and (2).

We first prove (1). We have the $S^1$-spectrum $\Sigma^\infty_{S^1}X_+\in \SH_{S^1}(k)$. By \cite[theorem 10]{Jardine2} we have
\[
\Pi_{r,n}(\Sigma^\infty_TX_+)=\colim_b\Pi_{r,n+b}\Sigma^\infty_{S^1}\Sigma^b_{\G_m}X_+
\]
so it suffices to see that $\Sigma^\infty_{S^1}\Sigma^b_{\G_m}X_+$ is topologically -1 connected for all $b\ge0$. By remark~\ref{rem:ConnectTopConnect}, we need only see that  $\Sigma^\infty_{S^1}\Sigma^b_{\G_m}X_+$ is -1 connected. This follows from  Morel's $S^1$-stable $\A^1$-connectedness theorem \cite[theorem 4.2.10]{MorelLec}.

For (2), lemma~\ref{lem:Plus}  and our assumption on $k$ imply that  $\SH(k)_\Q=\SH(k)_\Q^+$. By a result of  Morel (proved in detail by Cisinski-D\'eglise \cite[theorem 15.2.13]{CisDeg}), there is an equivalence of triangulated categories $\SH(k)_\Q^+\cong \DM(k)_\Q$, and hence
\begin{align*}
\Hom_{\SH(k)}(\Sigma^a_{S^1}\Sigma^b_{\G_m}\Sigma^\infty_TU_+,\Sigma^\infty_TX_+)_\Q&\cong
\Hom_{\SH(k)_\Q^+}((\Sigma^a_{S^1}\Sigma^b_{\G_m}\Sigma^\infty_TU_+)_\Q^+,(\Sigma^\infty_TX_+)_\Q^+) \\ 
&\cong\Hom_{\DM(k)_\Q}(M(U)(b)_\Q[a+b], M(X)_\Q),
\end{align*}
where $M:\Sm/k\to \DM(k)$ is the canonical functor.  

Since $M(U)(b)$ is compact in $\DM(k)$, we have
\[
\Hom_{\DM(k)_\Q}(M(U)(b)_\Q[a+b], M(X)_\Q)\cong
\Hom_{\DM(k)}(M(U)(b)[a+b], M(X))_\Q.
\]
As $X$ is smooth and projective, we may use duality (via, e.g., duality in Chow motives over $k$ and \cite[V, prop. 2.1.4]{VSF})
\begin{align*}
\Hom_{\DM(k)}(M(U)(b)[c], M(X))&\cong \Hom_{\DM(k)}(M(U\times X), \Z(d-b)[2d-c])\\
&=H^{2d-c}(U\times X, \Z(d-b)).
\end{align*}
But for all $Y\in \Sm/k$, $H^p(Y, \Z(q))=0$ for $q<0$ \cite[corollary 2]{VoevChow}, $p\in\Z$. Thus the presheaf
\[
U\mapsto \Hom_{\SH(k)}(\Sigma^a_{S^1}\Sigma^b_{\G_m}\Sigma^\infty_TU_+,\Sigma^\infty_TX_+)_\Q
\]
is zero for $b>d$, and hence the associated sheaf $\Pi_{a,b}(\Sigma^\infty_TX_+)_\Q$ is zero as well.
\end{proof}

\begin{rem} Suppose that $k$ has finite 2-cohomological dimension or that $\Char k>0$, so $\SH_\fin(k)\subset \SH_\cfin(k)$. We do not know if $\SH_\cfin(k)\subset \SH_\fin(k)$.
\end{rem}

\begin{rem} The statement of  the result of Cisinski-D\'eglise \cite[theorem 16.2.13]{CisDeg} cited above is not  the same as given here;  Cisinski-D\'eglise show that the $\A^1$-derived category over a base-scheme $S$, with $\Q$-coefficients, $D_{\A^1,\Q}(S)$, has plus part  $D_{\A^1,\Q}(S)^+$  equivalent to the category of ``Beilinson motives'' $\DM_\text{\textcyr{\tiny{B}}}(S)$. For $S$ geometrically unibranch, they show \cite[theorem 16.1.4]{CisDeg} that $\DM_\text{\textcyr{\tiny{B}}}(S)\cong \DM(S)_\Q$. 

Furthermore, the well-known adjunction between simplicial sets and chain complexes  extends to give an equivalence of $\SH(S)_\Q$ with $D_{\A^1,\Q}(S)$ (see e.g. \cite[\S 5.3.35]{CisDeg}). Putting these equivalences all together gives us the equivalence $\SH(k)_\Q^+\cong \DM(k)_\Q$ we use in the proof of proposition~\ref{prop:Vanishing1}.

Alternatively, one also can repeat the argument used by Cisinski-D\'eglise, replacing $D_{\A^1,\Q}(S)$ with $\SH(k)_\Q$, and $\DM_\text{\textcyr{\tiny{B}}}(S)$ with $\Ho\,M\Z\text{-}\Mod_\Q$, where $M\Z$ is the commutative monoid in symmetric motivic spectra over $k$ defined in \cite{RO}, and representing motivic cohomology in $\SH(k)$. This shows that the forgetful functor $\Ho\,M\Z\text{-}\Mod_\Q\to \SH(k)_\Q$ induces an equivalence  $\Ho\,M\Z\text{-}\Mod_\Q\to \SH(k)_\Q^+$. One can then use the theorem of R\"ondigs-{\O}stv{\ae}r \cite[theorem 1.1, discussion preceding theorem 1.2]{RO}, which gives an equivalence of $\Ho\,M\Z\text{-}\Mod_\Q$ with $\DM(k)_\Q$ for $k$ a perfect field. 
\end{rem}

\section{The proof of the convergence theorem}\label{sec:Finale}
 
The following  result combines with lemma~\ref{lem:Induction} to form the heart of the proof of our main theorem.

\begin{lem}\label{lem:Vanishing2} Let $k$ be a perfect field. Suppose that $k$ has finite cohomological dimension $D_k$. 
Let $F$ be a perfect field extension of $k$ with  $\trdim_kF\le d$. Let $E=\Omega^\infty_T\sE$ for some $\sE\in \SH_\cfin(k)$ and suppose that $\Pi_{a,*}E=0$ for $a<0$. Consider the   map
\[
\rho_n(f_ME):\pi_rf_n(f_ME)(F)\to \pi_rf_ME(F).
\]
Then, for all $r, M\ge0$,  $n>\max(D_k+d, d(\sE))$, we have
 \[
\rho_n(f_ME)(\Fil^{simp}_n(f_ME)\pi_rf_n(f_ME)(F))=0.
\]
\end{lem}

\begin{proof}  Let $A= \max(D_k+d, d(\sE))\ge 0$. By theorem~\ref{thm:GWSliceRevisited},  and the definition of the filtration $F^n_{MW^{Tr}}\pi_rf_ME(F)$, it suffices to show that, for every finite extension $F'$ of $F$, the product
\[
\cup: \Pi_{r,n}f_ME(F')\otimes K^{MW}_{n}(F')\to  \Pi_{r,0}f_ME(F')
\]
is zero if $n>A$.

We note that $f_ME\cong\Omega^\infty_Tf_M\sE$ for $M\ge0$ (lemma~\ref{lem:SliceIso}). Thus, by  lemma~\ref{lem:CohFinTate}, $(\Pi_{r,n}f_ME)_\Q=0$ for $n>\max(d(\sE),-1)$ and  the image of $\cup$ is the same as the subgroup generated by the images of the maps
\[
\cup_N:  \Pi_{r,n}f_ME(F')_{N-\text{tor}}\otimes K^{MW}_{n}(F')/N\to  \pi_rf_M E(F');\quad N\in \N.
\]
 
First suppose $\Char k\neq2$. By Morel's theorem \cite[theorem 5.3]{MorelWitt}, we have a cartesian diagram
\[
\xymatrix{
 K^{MW}_{n}(F')\ar[r]\ar[d]&K^M_n(F')\ar[d]\\
 I(F')^n\ar[r]&I(F')^n/I(F')^{n+1}
 }
 \]
But  as  $\trdim_kF'\le d$, we have  $cd_2(F)\le cd(F)\le D_k+d$ \cite[II, \S4.2, proposition 11]{Serre}. Thus, for $n>D_k+d$, lemma~\ref{lem:torsion}  tells us that $K^{MW}_{n}(F')=K^{M}_{n}(F')$. Furthermore, for $N$ prime to the characteristic, the Bloch-Kato conjecture\footnote{Proved by Rost and Voevodsky  \cite{VBK}.  For a presentation of some of Rost's results that go into the proof of Bloch-Kato, see, e.g., \cite{HW, SuslinJouk}.} gives the isomorphism
 \[
 K^{M}_{n}(F')/N\cong H^n_\et(F', \mu^{\otimes n}_N),
 \]
 and hence $K^{M}_{n}(F')/N=0$ for $n>D_k+d$. For $N=p=\Char(k)$, it follows from \cite[theorem 2.1]{BlochKato} that 
$K^{M}_{n}(F')/p=0$, 
since $F'$ is perfect. Thus, for $n>D_k+d$, $K^{MW}_{n}(F')/N=0$, and hence the image of $\cup$ is zero if $n>A$.

If $\Char k=2$, then as $F'$ is perfect,   each $u\in F'$ is a square. Thus, by \cite[proposition 3.13]{MorelA1},   the quotient map $K^{MW}_n(F')\to K^M_n(F')$
is  an isomorphism for $n\ge0$. The remainder of the discussion is the same.
\end{proof}

Relying on lemma~\ref{lem:Vanishing2},  here is the main step in the proof of theorem~\ref{IntroThm:Main}:

\begin{prop}\label{prop:MainInd}  Let $k$ be a perfect field of  finite cohomological dimension and let $p$ be the exponential characteristic of $k$. Then there is a function $N_k:\Z^3\to\Z$, with $N_k(e,r,d)\ge0$ for all $e,d,r$,  such that,  given an integer $e$,  and an $\sE$ in $\SH_\cfin(k)$ with $c(\sE)\ge -1$ and  $d(\sE)\le e$, we have
\[
\Fil_\Tate^{N}\Pi_{r,q}f_M\Omega^\infty_T\sE(F)[1/p]=0
\]
for all $r, M\in \Z$, all $N\ge N_k(e-q,r,d)+q$,  all $q\ge0$, and all field extension $F$ of $k$ with $\trdim_kF\le d$.
\end{prop}

\begin{proof}  For a field $F$, let $F^{per}$ be the perfect closure of $F$. By proposition~\ref{prop:InvertP}, $\Fil_\Tate^N\Pi_{a,b}\sG(F)[1/p]\cong\Fil_\Tate^N\Pi_{a,b}\sG(F^{per})[1/p]$ for all $\sG\in \SH(k)$, so we may replace $F$ with $F^{per}$ whenever necessary.  

   From our hypothesis   $c(\sE)\ge-1$ and lemma~\ref{lem:CohFinTate} follows $c(f_M\sE)\ge-1$, so $\Pi_{r,q}\Omega^\infty_Tf_M\sE=\Pi_{r,q}f_M\Omega^\infty_T\sE=0$ for $r<0$, $q\ge0$.  Thus, taking $N_k(e,r,d)=0$ for $r<0$, and all $d$ and $e$ settles the cases $r<0$, so we may proceed by induction on $r$, defining $N_k(e,r,d)$ recursively.  We omit the $[1/p]$ in the notation, using the convention that we invert  $p$ throughout the proof of this proposition.

Assume $r\ge0$ and that we have  defined  $N_k(e, j,d)$ for $j\le r-1$ (and all $e, d$) so that the proposition holds for $\Fil_\Tate^{N}\Pi_{j,q}f_M\Omega^\infty_T\sE(F)$, $j<r$; we assume in addition that $N_k(e, j,d)=0$ for $j<0$. Define $N_j(d,q; E):=N_k(e-q,j,d)+q$ for $j<r$. Then \eqref{eqn:IndAssump} is satisfied for $q, M\ge 0$, $m\ge N_j(d,q; E)$, $j=0,\ldots, r-1$.  By  lemma~\ref{lem:Induction}, it follows that for all $n\ge0$, all fields $F$ with $\text{tr.\,dim}_kF\le d$ and all 
$m\ge \max_{j=0}^{r-1}N_j(r-j+d, n; E)$ (or $m\ge0$ in case $r=0$), we have
\[
\pi_r(f_nf_mE)(F) =\Fil^{simp}_n(f_mE)\pi_r(f_nf_mE)(F).
\]
As the filtration $\Fil^{simp}_*(G)\pi_r(f_nG)(F)$ is functorial in $G$, this implies that 
\[
f_n(\rho_m(E))(\pi_r(f_nf_mE)(F))\subset \Fil^{simp}_n(E)\pi_r(f_nE)(F)
\]
for $n, m, F$ as above. 

 If  we  take $n\ge \max(D_k+d, d(\sE))+1$, then by lemma~\ref{lem:Vanishing2}  
\[
\im(\Fil^{simp}_n(E)\pi_r(f_nE)(F^{per})\xrightarrow{\rho_n(E)} \pi_r(E)(F^{per}))=0. 
\]
For $m\ge n$, we have $f_mf_n=f_m$ and   $\rho_n(E)\circ f_n(\rho_m(E))=\rho_m(E)$ (lemma~\ref{lem:SliceCompat}). Thus, for 
\[
n= \max(D_k+d, e)+1,\ m\ge \max_{j=-1}^{r-1}N_j(r-j+d, n; E),
\]
we have
\[
0=\rho_n(E)\circ f_n(\rho_m(E))(\pi_r(f_mf_nE)(F^{per}))=\rho_m(\pi_r(f_mE)(F^{per}))\subset  \pi_r E(F^{per}).
\]
Let $n= \max(D_k+d, e)+1$ and define $N_k(e,r,d):=n+ \max_{j=-1}^{r-1}N_k(e-n,j, r-j+d)$. Using the isomorphism $F^N_\Tate \pi_rE(F)\cong F^N_\Tate \pi_rE(F^{per})$, we thus have 
\[
F^N_\Tate \pi_rE(F)=0
\]
for $N\ge N_k(e,r,d)$  and $F$ a field extension of $k$ with $\trdim_kF\le d$.

Fix $q\ge0$ and let $\sE'=\Omega^q_T\sE$, $E'=\Omega^\infty_T\sE'$. Then $\Pi_{a,b}\sE'\cong \Pi_{a,b+q}\sE$ for all $b\in\Z$, hence 
$c(\sE')\ge-1$. Also, $d(\sE')=d(\sE)-q\le e-q$, and we may apply our result with $E'$ replacing $E$, and $e-q$ replacing $e$. Since
 $\pi_rE'=\Pi_{r,q}E$ and $F^{N-q}_\Tate \pi_rE'=F^N_\Tate\Pi_{r,q}E$ (lemma~\ref{lem:DegreeShift}), we thus have 
\[
F^N_\Tate \Pi_{r,q}E(F)=0
\]
for $N\ge N_k(e-q,r,d)+q$, $q\ge0$, and $F$ as above.

Now take an integer $M$. Then    $c(f_M\sE)\ge-1$   and $d(f_M\sE)\le d(\sE)$ (by lemma~\ref{lem:CohFinTate}), so the same result holds for $E_M:=\Omega_T^\infty f_M\sE$. By lemma~\ref{lem:SliceIso}, $E_M\cong f_ME$ , so
\[
F^N_\Tate \Pi_{r,q}f_ME(F)[1/p]=0
\]
for $q\ge0$, $N\ge N_k(e-q,r,d)+q$, and all field extensions $F$ of $k$ with $\trdim_kF\le d$. As $M$ was arbitrary, the induction thus goes through, completing the proof.
\end{proof}

Now for the proof of the main theorem. For $x\in X\in \Sm/k$, and $\sF$ a sheaf on $\Sm/k_\Nis$, we let $\sF_x$ denote the Nisnevich stalk of $\sF$ at $x$.

 \begin{thm}\label{thm:Main} Let $k$ be a perfect  field of finite cohomo\-logical dimension. Let $p$ be the exponential characteristic of $k$, let $x\in X\in \Sm/k$ and let $d=\dim_kX$. Let $\sE$ be in $\SH_\cfin(k)$.  Then for all integers $r, q$, and $M$,  we have
\[
(\Fil_\Tate^n\Pi_{r,q}f_M\sE)_x[1/p]=0
\]
for $n\ge N_k(d(\sE)-q,r-c(\sE)-1,d)+q$, where $N_k$ is the integer-valued function given by proposition~\ref{prop:MainInd}.
 \end{thm}
 
 As $\SH(k)_\fin$ is a subcategory of $\SH(k)_\cfin$ in case $k$ has finite cohomological dimension (proposition~\ref{prop:Vanishing1}(3)), our  main theorem~\ref{IntroThm:Main} follows immediately from theorem~\ref{thm:Main}, with $N(\sE, r,d,q):=N_k(d(\sE)-q,r-c(\sE)-1,d)+q$.

\begin{proof}[Proof of theorem~\ref{thm:Main}] As before, we invert $p$ throughout the proof and omit the ``$[1/p]$'' from the notation.

The proof of \cite[lemma 6.1.4]{MorelConn} shows that, for $x\in X\in\Sm/k$,  $X_x:=\Spec \sO_{X,x}$ and $U\subset X_x$ open, the map
$X_x\to X_x/U$
in $\sH(k)$ is equal to the map sending $X_x$ to the base-point of $X_x/U$. This implies that for any $\sF\in \SH(k)$, $a,b\in\Z$, the restriction map
\[
[\Sigma^a_{S^1}\Sigma^b_{\G_m}\Sigma^\infty X_{x+},\sF]_{\SH(k)}\to
[\Sigma^a_{S^1}\Sigma^b_{\G_m}\Sigma^\infty U_+,\sF]_{\SH(k)}
\]
is injective. Passing to the limit over $U$, this shows that the restriction map
\[
\Pi_{a,b}(\sF)_x\to \Pi_{a,b}(\sF)(k(X))
\]
is injective. From this it easily follows that the restriction map 
\[
(\Fil_\Tate^n\Pi_{r,q}f_M\sE)_x\to (\Fil_\Tate^n\Pi_{r,q}f_M\sE)(k(X))
\]
is injective.

Thus,  it suffices to show that $\Fil_\Tate^n\Pi_{r,q}f_M\sE(F)=0$ for $M\in\Z$, for $n\ge N(d(\sE)-q,r-c(\sE)-1,d)+q$, and  for all finitely generated fields $F$ over $k$ with $\trdim_kF\le d$.

As 
\[
\Fil_\Tate^n\Pi_{r+p,q}\Sigma^p_{S^1}\sE=\Fil_\Tate^n\Pi_{r,q}\sE
\]
we may replace $\sE$ with $\Sigma^{-c(\sE)-1}_{S^1}\sE$ and assume that $c(\sE)\ge-1$. Similarly, if $q<0$ we may  replace $\sE$ with $\Sigma^{-q}_{\G_m}\sE$, since 
\[
\Fil_\Tate^{n-q}\Pi_{r,0}f_{M-q}\Sigma^{-q}_{\G_m}\sE=\Fil_\Tate^n\Pi_{r,q} f_M\sE
\]
and $d(\Sigma^{-q}_{\G_m}\sE)=d(\sE)-q$.  This reduces us to the case $q\ge0$ and $c(\sE)\ge -1$. Letting $E=\Omega^\infty_T\sE$, we have 
\[
\Fil_\Tate^n\Pi_{r,q}f_M\sE=\Fil_\Tate^n\Pi_{r,q}\Omega^\infty_Tf_M\sE=\Fil_\Tate^n\Pi_{r,q}f_ME
\]
for all $M\in\Z$, $q\ge 0$, $n\ge 0$, so the result follows from proposition~\ref{prop:MainInd}.
\end{proof}

\begin{cor}\label{cor:Main} Take $k$, $x$, $d$ and $\sE$  as in theorem~\ref{thm:Main}. Then for all $q\in \Z$, $(\Pi_{r,q}f_M\sE)_x[1/p]=0$
for all $M\ge N_k(d(\sE)-q,r-c(\sE)-1,d)+q$.
\end{cor}

\begin{proof} Indeed, for $M\ge n$, $f_nf_M\cong f_M$, hence $\Fil_\Tate^n\Pi_{r,q}f_M\sE=\Pi_{r,q}f_M\sE$. The result thus follows from 
theorem~\ref{thm:Main}. 
\end{proof}

\begin{rem} Althoug $N_k(e,r,d)$ is defined  recursively, it has a simple expression: For $D_k+d\ge 0$, 
\[
N_k(e,r,d)=\begin{cases} (r+1)(D_k+d)+\frac{1}{2}(r+1)(r+2)&\text{ if } e\le D_k+d, r\ge0 \\
e+ r(D_k+d)+\frac{1}{2}(r+1)(r+2)&\text{ if } e> D_k+d, r\ge0\\
0&\text{ if } r<0.\end{cases}
\]
In particular, $N_k(e,r,d)$ is an increasing function in each variable (assuming $D_k+d\ge0$). Thus, one can apply theorem~\ref{thm:Main} or corollary~\ref{cor:Main} even if one only has an upper bound for $d(\sE)$ and a lower bound for $c(\sE)$.

For instance, for $Y\in \Sm/k$, we have $d(\Sigma^\infty_TY_+)\le \dim_kY$ and $c(\Sigma^\infty_TY_+)\ge-1$. Thus, for  $x\in X\in \Sm/k$, $r\ge0$, $M\in\Z$, we have   $(F^N_\Tate \Pi_{r,q}f_M\Sigma^\infty_TY_+)_x=0$ for 
\[
N\ge  
\max(\dim_kY-q,  D_k+\dim_kX)+r(D_k+\dim_kX)+\frac{1}{2}(r+1)(r+2)+q. 
\]
We do not know if any of these bounds are sharp. 
\end{rem}

\begin{appendix}
\section{Norm maps}\label{sec:Norm} Suppose our perfect base-field $k$ has characteristic $p>0$. For an abelian group $A$, we write $A'$ for $A\otimes_\Z\Z[1/p]$. Our task in this section is to prove 

\begin{prop}\label{prop:InvertP} 1. Suppose that $k$ is a perfect field of characteristic $p>0$. Take $\sE\in \SH(k)$ and let $\alpha:F\to L$ be a purely inseparable extension of finitely generated fields over $k$. Then the map
\[
\alpha^*:\Pi_{a,b}\sE(F)'\to \Pi_{a,b}\sE(L)'
\]
is injective.
\\
2. Suppose that $k$ is a finite field. Take $\sE\in \SH(k)$ and let $\alpha:k\to k'$ be a finite extension of degree $n$, let $F$ be a field extension of $k$, and let $F'=F\otimes_kk'$. Then the kernel of 
\[
\alpha^*:\Pi_{a,b}\sE(F)\to \Pi_{a,b}\sE(F')
\]
is $n$-torsion.
\end{prop}

As one would expect, we construct a quasi-inverse to $\alpha^*$ by constructing transfers. We first discuss (1). Take $a\in F^\times\setminus (F^\times)^p$ and consider the extension $F_a:=F(a^{1/p})$.  We have the corresponding closed point $p(a)$ of $\A^1_F$ defined by the homogeneous ideal $(X^p-a)\subset F[X]$. This gives us the closed point $p(a)$ of $\P^1_F$ via the standard open immersion $x\mapsto (1:x)$ of $\A^1$ into $\P^1$.

Suppose $F=k(U)$ for some finite type $k$-scheme $U$. Since $k$ is perfect, $U$ has a dense open subscheme smooth over $k$; shrinking $U$ and changing notation, we may assume that $U$ is affine and smooth over $k$,  and also that $a$ is a global unit on $U$. Let $V\subset \A^1\times U$ be the closed subscheme defined by $X^p-a$. Then $V$ is reduced and irreducible (since $a\not\in (F^\times)^p$) and is finite over $U$. Shrinking $U$ again, we may assume that $V$ is also smooth and affine over $k$. Let $I_V\subset k[U][t]$ be the ideal defining $V$ in $\A^1\times U$.

Using $X^p-a$ as generator for $I_V/I_V^2$, we have the Morel-Voevodsky purity isomorphism \cite[theorem 3.2.23]{MorelVoev}
\[
\A^1\times U/(\A^1\times U\setminus V)\cong \P^1_V.
\]
Combining with the excision isomorphism $\A^1\times U/(\A^1\times U\setminus V)\cong
\P^1\times U/(\P^1\times U\setminus V)$ and passing to the limit over open subschemes of $U$
 gives us the sequence of maps of  pro-objects in $\sH_\bullet(k)$:
\[
(\P^1_F,\infty)\to \P^1_F/\P^1_F\setminus\{p(a)\}\cong (\P^1_{F_a},\infty);
\]
we denote the composition by $\Tr_{p(a)/F}$. Passing to $\SH(k)$ gives us the morphism
\[
\Tr_{p(a)/F}: \mS_k\wedge\Spec F_+\to  \mS_k\wedge p(a)_+=\mS_k\wedge \Spec F_{a+}
\]
where we consider these objects as pro-objects in $\SH(k)$. Composing with the map induced by the structure morphism $\pi_a:p(a)\to \Spec F$ gives us the endomorphism $\pi_{a}\circ \Tr_{p(a)/F}$ of $\mS_k\wedge\Spec F_+$. 

Proposition~\ref{prop:InvertP}(1) is an immediate consequence of

\begin{lem} After inverting $p$, the  the endomorphism 
\[
\pi_{a}\circ \Tr_{p(a)/F}:\mS_k\wedge\Spec F_+\to \mS_k\wedge\Spec F_+
\]
 is an isomorphism.
\end{lem}

\begin{proof}
We consider a deformation of  $\Tr_{F_a/F}$. Let $P(a)\subset \Spec F[t][X]$ be the closed subscheme defined by the ideal $(X^p+t(X+1)+(t-1)a)$. We have
\[
d(X^p+t(X+1)+(t-1)a)=(X+a+1)dt+tdX\in \Omega_{F[t][X]/F}.
\]
Thus the singular locus of $P(a)$ (over $F$) is given by 
\[
(X+a+1=t=0)\cap(X^p+t(X+1)+(t-1)a=0); 
\]
since $a\not\in (F^\times)^p$, the singular locus is empty, i.e., $P(a)$ is smooth over $F$. Using $X^p+t(X+1)+(t-1)a$ as the generator for $I_{P(a)}/I^2_{P(a)}$, we have as above the map 
\[
\Tr_{P(a)/F[t]}: \P^1_k\wedge\Spec F[t]_+\to  \P^1_k\wedge P(a)_+
\]
of pro-objects in  $\sH_\bullet(k)$, and the endomorphism
\[
\pi_{P(a)/F[t]}\circ \Tr_{P(a)/F[t]}: \mS_k\wedge\Spec F[t]_+\to  \mS_k\wedge\Spec F[t]_+
\]
of pro-objects in $\SH(k)$. 

Setting $t=1$ gives us the closed subscheme  $p_0$ of $\Spec F[X]$ defined by the ideal $(X^p+X+1)$. Clearly $\pi_0:p_0\to \Spec F$ is finite and \'etale of degree $p$. Using $X^p+X+1$ as generator of $I_{p_0}/I^2_{p_0}$ gives us the map 
\[
\Tr_{p_0/F}: (\P^1_k,\infty)\wedge\Spec F_+\to  (\P^1_k,\infty)\wedge p_{0+}
\]
of pro-objects in  $\sH_\bullet(k)$, and the endomorphism
\[
\pi_{0}\circ \Tr_{p_0/F}: \mS_k\wedge\Spec F_+\to  \mS_k\wedge\Spec F_+
\]
of pro-objects in $\SH(k)$. 

The map $\pi_{P(a)/F[t]}\circ \Tr_{P(a)/F[t]}$ thus gives us an $\A^1$-homotopy between the maps $\pi_{a}\circ \Tr_{p(a)/F}$ and $\pi_{0}\circ \Tr_{p_0/F}$. Thus, these maps are equal in $\SH(k)$, and it suffices to show that 
$\pi_{0}\circ \Tr_{p_0/F}$ is an isomorphism after inverting $p$.

As $X^p+X+1$ has coefficients in $\F_p$, $\pi_{0}\circ \Tr_{p_0/F}$ arises as base-extension from the similarly defined map 
\[
\pi_{0}\circ \Tr_{p_0/\F_p}: \mS_{\F_p}\to  \mS_{\F_p}
\]
in $\SH(\F_p)$, that is, from the corresponding element $[\pi_{0}\circ \Tr_{p_0/\F_p}]\in [\mS_{\F_p},\mS_{\F_p}]_{\SH(\F_p)}$. 
 
By Morel's theorem  $[\mS_{\F_p},\mS_{\F_p}]_{\SH(\F_p)}\cong \GW(\F_p)$ (see \cite[lemma 3.10,  corollary 6.41]{MorelA1}). Since $[\F_p(p_0):\F_p]=p$, it follows that the image of $\pi_{0}\circ \Tr_{p_0/\F_p}\in [\mS_{\F_p},\mS_{\F_p}]_{\SH(\F_p)}$ under the rank homomorphism $\GW(\F_p)\to \Z$ is $p$. Since the augmentation ideal $I\subset \GW(\F_p)$ is nilpotent (in fact $I^2=0$), it follows that $\pi_{0}\circ \Tr_{p_0/\F_p}$ is a unit in $\GW(\F_p)[1/p]$, completing the proof of (1).

The proof of proposition~\ref{prop:InvertP}(2) is similar. Write $k'=k(a)$, let $f(x)\in k[x]$ be the minimal polynomial of $a$ and let $p\in \A^1_k$ be the closed point defined by the ideal $(f(x))$. Using $f(x)$ as the generator for $I_p/I_p^2$, the Morel-Voevodsky purity isomorphism \cite[\hbox{\it{loc. cit.}}]{MorelVoev} defines the map
\[
\Tr_{p/k}: (\P^1_k,\infty)\to  (\P^1_k,\infty)\wedge p_+
\]
 in  $\sH_\bullet(k)$, and the endomorphism
\[
\pi_p\circ \Tr_{p/k}: \mS_k\to  \mS_k 
\]
 in $\SH(k)$. We can use Morel's theorem again and identify $\End(\mS_k)$ with $\GW(k)$. Under this identification, $\pi_p\circ \Tr_{p/k}$ corresponds to an element $\gamma\in \GW(k)$, mapping to the degree $n$ under the rank homomorphism $\GW(k)\to \Z$. Since the augmentation ideal $I\subset \GW(k)$ is nilpotent, this implies that there is an element $\beta\in\GW(k)$ with $\beta\cdot \gamma=n$. Viewing $\beta$ as an endomorphism of $\mS_k$, this gives $\pi_p\circ(\Tr_{p/k}\circ\beta)=n\times\id$. As $\alpha^*=\pi_p^*\wedge\id_{\Spec F}$ on $\Pi_{a,b}\sE(F)$, this implies $n\cdot\ker \alpha^*=0$, completing the proof.
\end{proof}

 \section{Inverting integers in a triangulated category} \label{sec:Localization}
 
 We collect some notations and elementary facts concerning the localization of a triangulated category with respect to a multiplicatively closed subset of $\Z\setminus\{0\}$. If every object of the triangulated category is compact, there is no issue about what such a localization is and how to define it, however for ``large" triangulated categories, there are a number of definitions possible, as well as some issues involving extending adjoints. 
 
Although this material is quite elementary,  these facts on localization do not appear to be in the literature in the setting of compactly generated triangulated categories admitting arbitrary coproducts. Some of this material appears in an appendix to S. Kelly's doctoral thesis \cite{KellyThesis}; I am grateful to him for pointing out these results to me. 
 
 In what follows,  $\sT$ will be a triangulated category admitting arbitrary small coproducts.

 \begin{rem}\label{rem:TriangCats}
Suppose $\sR$ is a localizing subcategory of $\sT$ with a set of compact generators. Then (see e.g. \cite[construction 1.6]{NeemanThBousRav}), the Verdier localization $q:\sT\to \sT/\sR$ exists and admits a  right adjoint $r:\sT/\sR\to \sT$. $r$ induces an equivalence of $\sT/\sR$ with $\sR^\perp$, this being the full subcategory of objects $X\in \sT$ such that $\Hom_\sT(A,X)=0$ for all $A\in \sR$. If $R$ is a set of compact generators  for $\sR$, then $\sR^\perp=R^\perp$.

Let $R$ be a set of compact objects in $\sT$ and $\sR$ the localizing subcategory of $\sT$ generated by $R$, that is, the smallest localizing subcategory of $\sT$ containing $R$. We recall from \cite[theorem 2.1]{NeemanGrothDual} that $R$ is a set of generators for $\sR$, that is, if $X$ is an object of $\sR$ such that $\Hom_\sR(A,X)=0$ for all $A\in R$, then $X\cong 0$. Furthermore, the subcategory $\sR^c$ of compact objects in $\sR$ is the thick subcategory of $\sR$ generated by $R$. 
\end{rem}

Let $S$ be a multiplicatively closed subset of $\Z\setminus\{0\}$ containing 1.  Call an object $X$ in $\sT$ {\em $S$-torsion} if $n\cdot\id_X=0$ for some $n\in S$. We let $\sT_{S\text{-tor}}$ be the localizing subcategory of  $\sT$ generated by the compact $S$-torsion objects of $\sT$ and let  $\sT_{S^{-1}\Z}$ denote the Verdier localization $\sT/\sT_{S\text{-tor}}$. If $S$ is the set of powers of some integer $n$, we write $\Z[\frac{1}{n}]$ for $S^{-1}\Z$,  $\sT_{n^\infty\text{-tor}}$ for $\sT_{S\text{-tor}}$ and $\sT[\frac{1}{n}]$ for $\sT_{S^{-1}\Z}$. For an object $A$ of $\sT$ write $A_{S^{-1}\Z}$ or $A[\frac{1}{n}]$ for the image of $A$ in $\sT_{S^{-1}\Z}$ or $\sT[\frac{1}{n}]$; for an abelian group or (pre)sheaf of abelian groups $M$, we write $M_{S^{-1}\Z}$ for $M\otimes_\Z S^{-1}\Z$. For $S=\Z\setminus\{0\}$, we will of course write $\Q$ for $S^{-1}\Z$.

If $\sT$ is a compactly generated tensor  triangulated category such that the tensor product of compact objects is compact,  then $\sT_{S\text{-tor}}$ is a tensor ideal (since $n\cdot\id_{A\otimes B}=(n\cdot\id_A)\otimes\id_B$) so $\sT_{S^{-1}\Z}$ inherits a tensor structure from $\sT$, and the localization functor $\sT\to \sT_{S^{-1}\Z}$ is an exact tensor functor of tensor triangulated categories. 

For $A$ an object in  $\sT$, let $A\otimes^L\Z/n$ denote an object fitting into a distinguished triangle
\[
A\xrightarrow{n\cdot\id}A\to A\otimes^L\Z/n\to A[1];
\]
this defines $A\otimes^L\Z/n$ up to non-unique isomorphism. In addition, $A\otimes^L\Z/n$ is an $n^2$-torsion object. 

\begin{lem} \label{lem:Localization} Let $S$ be a subset of $\Z\setminus\{0\}$ containing 1, and suppose $\sT$ has a set $C$ of compact generators.\\
1. The set $C_S:=\{A\otimes^L\Z/n\ |\ A\in C, n\in S\}$
is a set of compact generators for $\sT_{S\text{-tor}}$.\\
2. For $A$, $X$ objects in $\sT$ with $A$ compact, there is a canonical isomorphism
\[
\Hom_\sT(A,X)_{S^{-1}\Z}\cong \Hom_{\sT_{S^{-1}\Z}}(A_{S^{-1}\Z,} X_{S^{-1}\Z}).
\]
In addition, for $A$ compact in $\sT$, $A_{S^{-1}\Z}$ is compact in $\sT_{S^{-1}\Z}$, and $\sT_{S^{-1}\Z}$ is compactly generated, with $C^S:=\{A_{S^{-1}\Z}, A\in C\}$ a set of compact generators.\\
3. $\sT_{S\text{-tor}}$ is equal to the full subcategory $\sT(S)$ of $\sT$ with objects $X$ such that $\Hom_\sT(A,X)_{S^{-1}\Z}=0$ for all compact $A$ in $\sT$.\\
4. $\SH(k)_{S\text{-tor}}$ is the full subcategory of $\SH(k)$ with objects those $\sK$ such that $\Pi_{a,b}(\sK)_{S^{-1}\Z}=0$ for all $a,b\in\Z$.\\
5.  A map $f:\sE\to \sF$ in $\SH(k)$ becomes an isomorphism in $\SH(k)_{S^{-1}\Z}$ if and only if the induced map on the homotopy sheaves
\[
f_*\otimes\id:\Pi_{a,b}(\sE)_{S^{-1}\Z}\to \Pi_{a,b}(\sF)_{S^{-1}\Z}
\]
is an isomorphism for all $a,b\in\Z$.\\
6. $\SH(k)_{S\text{-tor}}$ is a tensor ideal, hence $\SH(k)_{S^{-1}\Z}$ is a tensor triangulated category and the localization functor $\SH(k)\to \SH(k)_{S^{-1}\Z}$ is an exact tensor functor.
\end{lem}

\begin{proof} For (1), clearly $C_S$ consists of compact objects of $\sT_{S\text{-tor}}$. If $\sX$ is $n$-torsion in  $\sT$, then $\sX$ is a summand of $\sX\otimes^L\Z/n$, hence the localizing subcategory of $\sT$ generated by  $C_S$ contains all the compact objects in $\sT_{S\text{-tor}}$, hence is equal to $\sT_{S\text{-tor}}$, proving (1). 

To prove (2), let $Y\xrightarrow{v} A\xrightarrow{u}Z\to Y[1]$ be a distinguished triangle in $\sT$ with $A$ compact and $Z$ in $\sT_{S\text{-tor}}$. Then $n\cdot u=0$ for some $n\in S$, so there is a map $w:A\to Y$ with $v\circ w=n\cdot\id_A$. Thus the category of maps $\{n\cdot\id:A\to A\}$ is cofinal in the category of maps $Y\to A$ with cone in $\sT_{S\text{-tor}}$, from which the isomorphism in (2) follows directly.  

As the localization functor $q:\sT\to \sT_{S^{-1}Z}$ admits a right adjoint, $q$ preserves all small coproducts; the isomorphism we have just proved shows that $A_{S^{-1}\Z}$ is compact if $A$ is compact. If $\Hom_\sT(A,X)_{S^{-1}\Z}=0$ for all $A\in C$, then clearly $X$ is in $\sT_{S\text{-tor}}$ hence $X_{S^{-1}\Z}=0$, completing the proof of (2). (3) follows immediately from (2).
 
 Clearly (5) follows from (4). To prove (4), let $\SH(k)_{S\text{-tor}}^*\subset \SH(k)$ be the full subcategory of objects $\sK$ as in (4). We recall that $\SH(k)$ has the set   of compact generators $\{\Sigma^a_{S^1}\Sigma^b_T\Sigma^\infty_TX_+\ |\ a,b\in \Z, X\in \Sm/k\}$.

 Take $\sK\in \SH(k)_{S\text{-tor}}^*$, $X\in \Sm/k$, and let $f:\Sigma^a_{S^1}\Sigma^b_T\Sigma^\infty_TX_+\to \sK$ be a morphism. Using the Gersten spectral sequence on $X$ and the assumption that $\sK$ is in $\SH(k)_{S\text{-tor}}^*$, we see there is an $n\in S$ such that $n\cdot f=0$. Thus $\SH(k)_{S\text{-tor}}^* \subset \SH(k)(S)$; the reverse inclusion follows from the definition of the homotopy sheaf $\Pi_{a,b}(\sK)$ as the Nisnevich sheaf associated to the presheaf
 \[
 X\mapsto [\Sigma^{a-b}_{S^1}\Sigma^b_T\Sigma^\infty_TX_+, \sK]_{\SH(k)}.
 \]
 By (3), this shows that $\SH(k)_{S\text{-tor}}^*=\SH_{S\text{-tor}}$.
 
For (6), we note that $\Sigma^\infty_TX_+\wedge \Sigma^\infty_TY_+\cong \Sigma^\infty_TX\times Y_+$. As the subcategory of compact objects of $\SH(k)$ is the thick subcategory generated by $C$ (see remark~\ref{rem:TriangCats}), it follows that the $\wedge$-product of compact objects is compact, which suffices to prove (6).
\end{proof}

\begin{rem} Lemma~\ref{lem:Localization}(4)-(6) hold with $\SH(k)$ replaced by $\SH_{S^1}(k)$, with the obvious modification in the statements.
\end{rem}

\begin{lem}\label{lem:LocAdjoint}  Let $L:\sT_1\to \sT_2$ be an exact functor of compactly generated triangulated categories, with right adjoint $R:\sT_2\to \sT_1$; we assume that $\sT_1$ and $\sT_2$ admit arbitrary small coproducts. Suppose that $L(A)$ is compact for $A$ in $\sT_1$ compact. Then\\
1. $R$ is compatible with small coproducts.\\
2. Let $S\subset \Z\setminus\{0\}$ be a multiplicatively closed subset. Then $(L, R)$ descends to an adjoint pair $L_{S^{-1}\Z}:\sT_{1S^{-1}\Z}\xymatrix{\ar@<3pt>[r]&\ar@<3pt>[l]}\sT_{2S^{-1}\Z}:R_{S^{-1}\Z}$.\\
3. Suppose that $L$ is the inclusion functor for a full triangulated subcategory $\sT_1$ of $\sT_2$. Then $L_{S^{-1}\Z}:\sT_{1S^{-1}\Z}\to \sT_{2S^{-1}\Z}$ is an isomorphism of $\sT_{1S^{-1}\Z}$ with its image in $\sT_{2S^{-1}\Z}$.
\end{lem}

\begin{proof} For (1),  it is easy to show that $L$ sends compact objects of $\sT_1$ to compact objects of $\sT_2$ if and only if $R$ preserves small coproducts; we leave the details of the proof to the reader.

For (2), if $C_2$ is a set of compact generators for $\sT_2$, and $A$ is in $C_2$, then for $n\in S$,  $R(A)\otimes^L\Z/n\cong R(A\otimes^L\Z/n)$, hence $R$ maps $C_{2S}$ to $\sT_{1S\text{-tor}}$. As $R$ is compatible with small coproducts, it follows that $R(\sT_{2S\text{-tor}})\subset \sT_{1S\text{-tor}}$. As $L$ is a left adjoint, $L$ is compatible with small coproducts, so the same argument shows that $L(\sT_{1S\text{-tor}})\subset \sT_{2S\text{-tor}}$, giving the induced functors on the localizations
\[
L_{S^{-1}\Z}:\sT_{1S^{-1}\Z}\xymatrix{\ar@<3pt>[r]&\ar@<3pt>[l]}\sT_{2S^{-1}\Z}:R_{S^{-1}\Z}.
\]
To show that $(L_{S^{-1}\Z}, R_{S^{-1}\Z})$ is an adjoint pair, we may replace $\sT_{iS^{-1}\Z}$ with the equivalent full subcategory $C_{iS}^\perp$ of $\sT_i$, $i=1,2$. $R$ maps $C_{2S}^\perp$ to $C_{1S}^\perp$, and $R_{S^{-1}\Z}$ is identified with the functor $R_S:C_{2S}^\perp\to C_{1S}^\perp$ induced from $R$. Letting $f:\sT_2\to C_{2S}^\perp$ be the left adjoint to the inclusion $i_2:C_{2S}^\perp\to \sT_2$, $L_{S^{-1}\Z}$ is identified with the restriction $L_S$ of $f\circ L$ to $C_{1S}^\perp$. Letting $i_1: C_{1S}^\perp\to \sT_1$ be the inclusion, we have
\begin{multline*}
\Hom_{C_{2S}^\perp}(L_S(X), Y)\cong \Hom_{\sT_2}(L(i_1X), i_2Y)\\
\cong \Hom_{\sT_1}(i_1X, R(i_2Y))=\Hom_{C_{1S}^\perp}(X, R_S(Y))
\end{multline*}
for $X\in C_{1S}^\perp$, $Y\in C_{2S}^\perp$.

For (3), as $L$ and $L_{S^{-1}\Z}$ are the same on objects, we need only show that $L_{S^{-1}\Z}$ is fully faithful. Let $X$, $Y$ be objects of $\sT_1$. If $X$ is compact,  then $L(X)$ is compact by assumption, and thus
\begin{multline*}
\Hom_{\sT_{2S^{-1}\Z}}(L_{S^{-1}\Z}(X_{S^{-1}\Z}), L_{S^{-1}\Z}(Y_{S^{-1}\Z}))= \Hom_{\sT_{2S^{-1}\Z}}(L(X)_{S^{-1}\Z}, L(Y)_{S^{-1}\Z})\\
\cong \Hom_{\sT_2}(L(X), L(Y))_{S^{-1}\Z}
\cong  \Hom_{\sT_1}(X, Y)_{S^{-1}\Z}\cong \Hom_{\sT_{1S^{-1}\Z}}(X_{S^{-1}\Z}, Y_{S^{-1}\Z}).
\end{multline*}
Furthermore, as $L_{S^{-1}\Z}$  is a left adjoint, it preserves small coproducts, hence for fixed $Y$, the full subcategory $\sT_{1S^{-1}\Z}^Y$ of $X_{S^{-1}\Z}$ in $\sT_{1S^{-1}\Z}$ such that 
\[
\Hom_{\sT_{2S^{-1}\Z}}(L_{S^{-1}\Z}(X_{S^{-1}\Z}), L_{S^{-1}\Z}(Y_{S^{-1}\Z}))\cong \Hom_{\sT_{1S^{-1}\Z}}(X_{S^{-1}\Z}, Y_{S^{-1}\Z})
\]
is a localizing subcategory of $\sT_{1S^{-1}\Z}$ containing the objects $A_{S^{-1}\Z}$ for $A$ a compact object of $\sT_1$. As these form a set of compact generators for $\sT_{1S^{-1}\Z}$, we see that 
$\sT_{1S^{-1}\Z}^Y=\sT_{1S^{-1}\Z}$, hence $L_{S^{_1}\Z}$ is fully faithful.
\end{proof}

\begin{ex} \label{ex:LocSlice} We consider the example of $i_n:\Sigma^n_T\SH^\eff(k)\xymatrix{\ar@<3pt>[r]&\ar@<3pt>[l]}\SH(k):r_n$. Lemma~\ref{lem:LocAdjoint} gives us the full subcategory $\Sigma^n_T\SH^\eff(k)_{S^{-1}\Z}$ of $\SH(k)_{S^{-1}\Z}$ with inclusion functor $i_{nS^{-1}\Z}$,   the adjoint pair of functors 
\[
 i_{nS^{-1}\Z}:\Sigma^n_T\SH^\eff(k)_{S^{-1}\Z}\xymatrix{\ar@<3pt>[r]&\ar@<3pt>[l]}\SH(k)_{S^{-1}\Z}:r_{nS^{-1}\Z},
 \]
the truncation functor $f_{nS^{-1}\Z}:=i_{nS^{-1}\Z}\circ r_{nS^{-1}\Z}$, and for $\sE\in \SH(k)$, the canonical isomorphism
\[
(f_n\sE)_{S^{-1}\Z}\cong f_{nS^{-1}\Z}(\sE_{S^{-1}\Z}).
\]
In addition, by lemma~\ref{lem:Localization}, $\Sigma^n_T\SH^\eff(k)_{S^{-1}\Z}$ is equal to the localizing subcategory of $\SH^\eff(k)_{S^{-1}\Z}$ generated by the set of compact objects $(\Sigma^q_T\Sigma^\infty_TX_+)_{S^{-1}\Z}$, for $q\ge n$ and $X\in \Sm/k$. The analogous results hold for $\SH_{S^1}(k)$. 
\end{ex}
\end{appendix}

\end{document}